\documentclass{amsart}
\usepackage[utf8]{inputenc}
\usepackage{mathrsfs}
\usepackage[all,cmtip]{xy} 
\usepackage{amssymb}
\usepackage{tikz}
\usepackage{caption}
\usepackage{bookmark}

\newtheorem*{prop*}{Proposition}
\newtheorem*{theorem*}{Theorem}
\newtheorem{theorem}{Theorem}[section]
\newtheorem{definition}[theorem]{Definition}

\newtheorem{prop}[theorem]{Proposition}
\newtheorem{lemma}[theorem]{Lemma}

\DeclareMathOperator{\Map}{Map}
\DeclareMathOperator{\g}{\mathscr{G}}

\DeclareMathOperator{\gq}{\mathscr{G}_Q}
\DeclareMathOperator{\Mapq}{\Map_{Q}}
\DeclareMathOperator{\Mappq}{\Map^*_{Q}}

\DeclareMathOperator{\ev}{ev}
\DeclareMathOperator{\id}{id}
\DeclareMathOperator{\sw}{sw}
\DeclareMathOperator{\he}{he}
\DeclareMathOperator{\Mape}{\Map_{\mathbb{Z}_2}}
\DeclareMathOperator{\Mappe}{\Map^*_{\mathbb{Z}_2}}

\DeclareMathOperator{\bin}{\sigma}
\DeclareMathOperator{\tin}{\tilde \sigma}

\newcommand{\pullbackcorner}[1][dr]{\save*!/#1-1.2pc/#1:(-1,1)@^{|-}\restore}

\newcommand{\leqnomode}{\tagsleft@true}
\newcommand{\reqnomode}{\tagsleft@false}

\definecolor{colour1}{rgb}{0,0,1}

\begin{document}
\author{Michael West}
\address{School of Mathematics, Uni.\ of Southampton, Salisbury Rd, SO17 1BJ}
\email{michael.west@soton.ac.uk}
\title{Homotopy Decompositions of Gauge Groups over Real Surfaces}
\begin{abstract}
 We analyse the homotopy types of gauge groups of principal $U(n)$-bundles associated to pseudo Real vector bundles in the sense of Atiyah \cite{atiyahktheoryandreality}.  We provide satisfactory homotopy decompositions of these gauge groups into factors in which the homotopy groups are well known. Therefore, we substantially build upon the low dimensional homotopy groups as provided in \cite{biswasstablevectorbundles}.
\end{abstract}
\maketitle

\renewcommand{\arraystretch}{1.8}

\tableofcontents
 \section{Introduction}  \label{chapterintroduction}
 Recently, the topology of gauge groups over Real surfaces has received widespread interest due to their intimate ties with the moduli spaces of stable vector bundles (see  \cite{biswasstablevectorbundles} and \cite{schaffhausermodulispace}).  Indeed, there have been explicit calculations of some of the topological invariants of these gauge groups.  For instance, Real vector bundles over Real surfaces were originally classified in \cite{biswasstablevectorbundles} but more recently in \cite{arxivgeorgievatopogrealbundlepairs}. Cohomology calculations of the classifying spaces appeared in \cite{liuyangmills},\cite{bairdmodulispace} and \cite{arxivbairdcohommodulispace}.  Furthermore, some of the low dimensional homotopy groups were presented in \cite{biswasstablevectorbundles}. The purpose of this paper is to extend the calculations of these homotopy groups by providing homotopy decompositions of the gauge groups into products of known factors.
 
 In the coming section, we define our objects of interest and state their classification results.  We go on to state the results of this paper, and then proofs are provided in Section \ref{chapterresults}.  In Section \ref{appendixa}, we present tables of homotopy groups, and compare them to those provided in \cite{biswasstablevectorbundles} in which we highlight a discrepancy.

\subsection*{Acknowledgements}
I would like to thank Prof.~ Tom Baird and the referee of this paper for suggesting the proofs of Propositions \ref{propunpointedimmediatefromnonequiv} and \ref{propunpointedcomponentequiv}.  Additionally, I further thank Prof.~ Baird for a suggested reformulation of Theorem \ref{thme} and for providing interesting conversations surrounding the topics in this paper. I would also like to thank Prof.~ Stephen Theriault for suggesting this project, along with his continued support and guidance.

 \subsection{Definitions}

 The pair $(X,\sigma)$, where $X$ is a compact connected Riemann surface, and $\sigma$ is an antiholomorphic involution, will be called a \textit{Real surface}.  
 
To a Real surface $(X,\sigma)$ we associate the following triple $(g(X),r(X),a(X))$ where
\begin{itemize}
\item $g(X)$ is the genus of $X$;
\item $r(X)$ is number of path components of the fixed set $X^\sigma$;
\item $a(X)=0$ if $X/\sigma$ is orientable and $a(X)=1$ otherwise.
\end{itemize}

We note that the path components of $X^\sigma$ are each homeomorphic to $S^1$.  The following classification of Real surfaces was studied in \cite{WKlein}.

\begin{theorem}[Weichold] \label{thmweichold} Let $(X,\sigma)$ and $(X',\sigma')$ be Real surfaces then there is a isomorphism $X\rightarrow X'$ (in the category of Real surfaces) if and only if 
\[(g(X),r(X),a(X))=(g(X'),r(X'),a(X')).\]
Furthermore, if a triple $(g,r,a)$ satisfies one of the following conditions
\begin{enumerate}
\item if $a=0$, then $1\leq r \leq g+1$ and $r\equiv (g+1) \mod 2$;
\item if $a=1$, then $0\leq r \leq g$;
\end{enumerate}
then there is a Real surface $(X,\sigma)$ such that $(g,r,a)=(g(X),r(X),a(X))$. \qed
\end{theorem}
\noindent Therefore a Real surface $(X,\sigma)$ is completely determined by its triple $(g,r,a)$ which we call the \textit{type} of the Real surface.

Let $\pi \colon P\rightarrow X$ be a principal $U(n)$-bundle over the underlying Riemann surface $X$ of the Real surface $(X,\sigma)$. A \textit{lift} of $\sigma$ is a map $\tilde{\sigma}\colon P\rightarrow P$ satisfying 
\begin{enumerate}
\item $\bin\pi = \pi \tin$;
\item $\tilde{\sigma}(p\cdot g)=\tilde{\sigma}(p)\cdot \overline{g}$ for all $p \in P, g \in U(n)$;
\end{enumerate} 
where $\overline g$ represents the entry-wise complex conjugate of $g\in U(n)$. We remark that, due to property $2$ of a lift, the fixed point set $P^{\tin}$ has the structure of a principal $O(n)$-bundle over the real points $X^\sigma$.

Let $\tin$ be a lift then we say that $(P,\tin)\rightarrow (X,\bin)$ is a \textit{Real principal $U(n)$-bundle} (or \textit{Real bundle}) if $\tin$ further satisfies
\begin{enumerate}
\item[$3.$] $\tilde{\sigma}^2(p)=p \mbox{ for all } p\in P ;$ 
\end{enumerate}
or if $n$ is even we say that $(P,\tin)\rightarrow (X,\bin)$ is a \textit{Quaternionic principal $U(n)$-bundle} (or \textit{Quaternionic bundle}) if $\tin$ satisfies
\begin{enumerate}
\item[$3'.$] $\tilde{\sigma}^2(p)=p\cdot (-I_n) \mbox{ for all } p\in P .$ 
\end{enumerate}
where $I_n$ represents the $n\times n$ identity matrix.  Such bundles were classified in \cite{biswasstablevectorbundles}.

\begin{prop}[Biswas, Huisman, Hurtubise]\label{proprealbundleclass}
Let $(X, \sigma)$ be a type $(g,r,a)$ Real surface and denote its fixed components by $X_i$ for $1\leq i \leq r$. Then Real principal $U(n)$-bundles $(P,\tilde\sigma)\rightarrow (X,\sigma)$ are classified by the first Stiefel-Whitney classes of the restriction to bundles $P_i\rightarrow X_i$ over the fixed components
\[\omega_1(P_i)\in H^1(X_i,\mathbb{Z}/2)\cong \mathbb{Z}/2\]
and the first Chern classes of the bundle over $X$
\[c_1(P)\in H^2(X,\mathbb{Z})\cong\mathbb{Z}\]
subject to the relation 
\begin{equation*}\label{eq:relation} c_1(P)\equiv\sum w_1(P_i) \mod (2).
 \end{equation*}

Furthermore, given any such characteristic classes then there is a Real principal $U(n)$-bundle that realises them. \qed
\end{prop}
We write \[(c,w_1,w_2,\dots,w_r):=(c_1(P),w_1(P_1),w_1(P_2),\dots,w_1(P_r))\]
and we will refer to the tuple $(c,w_1,w_2,\dots,w_r)\in \mathbb{Z}\times \prod_r\mathbb{Z}_2$ as the \textit{class} of the Real principal $U(n)$-bundle $(P,\tin)$.

\begin{prop}[Biswas, Huisman, Hurtubise]\label{propquatbundlesclass}
Let $(X,\bin)$ be a Real surface of type $(g,r,a)$ and let $n$ be even. Then Quaternionic principal $U(n)$-bundles $(P,\tin)\rightarrow (X,\bin)$ are classified by their first Chern class which must be even.  Furthermore, given any such Chern class then there is a Quaternionic principal $U(n)$-bundle that realises it. \qed
\end{prop}

We recall that we only cater for Quaternionic bundles of even rank.  However, a similar result for the case when $n$ is odd is also handled in \cite{biswasstablevectorbundles}.

Writing $c=c_1(P)$, we will therefore refer to $c\in 2\mathbb{Z}$ as the \textit{class} of the Quaternionic principal $U(n)$-bundle $(P,\tin)$.

Let $(P,\tin)\rightarrow (X,\bin)$ be a Real or Quaternionic principal $U(n)$-bundle.  An \textit{automorphism} of $(P,\tin)$ is a $U(n)$-equivariant map $\phi \colon P\rightarrow P$ such that the following diagrams commute
\[ \begin{gathered}\xymatrix{ P \ar[r]^-\phi \ar[d] & P \ar[d] \\ X \ar[r]^-{\id_{X}} & X}\end{gathered} \hspace{1cm} \mbox{ and } \hspace{1cm} \begin{gathered}\xymatrix{ P \ar[r]^-\phi \ar[d]^{\tin} & P \ar[d]^-{\tin} \\ P\ar[r]^-{\phi} & P.}\end{gathered}\]

Let $\Map(P,P)$ denote the set of self maps of $P$ endowed with the compact open topology.

\begin{definition} The \textit{(unpointed) gauge group} $\g(P,\tin)$ is the subspace of $\Map(P,P)$ whose elements are automorphisms of $(P,\tin)$.

\end{definition}

It will be convenient to provide decompositions for certain subspaces of the gauge group.  

\begin{definition} \label{defpgg} Choose a basepoint $*_X$ of $(X,\bin)$ such that $\bin(*_X)=*_X$ if $r>0$. Then the \textit{(single)-pointed gauge group} $\g^*(P,\tilde\sigma)$ consists of the elements of $\g(P,\tilde\sigma)$ that restrict to the identity above $*_X$.
\end{definition}

Another pointed gauge group of interest was considered in \cite{biswasstablevectorbundles}.  Let $(X,\bin)$ be a Real surface of type $(g,r,a)$, then for each $1\leq i\leq r$ choose a designated point $*_i$ contained in the fixed component $X_i$. Further, if $a=1$, choose another designated point $*_{r+1}$ that is not fixed by the involution. By convention, we choose $*_1$ to be $*_X$ as in Definition \ref{defpgg}.
\begin{definition}\label{defragg} The \textit{$(r+a)$-pointed gauge group} $\g^{*(r+a)}(P,\tilde\sigma)$ consists of the elements of $\g(P,\tilde\sigma)$ that restrict to the identity above these $(r+a)$ designated points of $(X,\bin)$.  
\end{definition}

\subsection{Main Results for Real Bundles}\label{sectionmainresultsforrealbundles}
In this section, we aim to present the main results pertaining to homotopy decompositions of gauge groups of Real principal $U(n)$-bundles. To ease notation we will sometimes use the following
\begin{itemize}
\item $\g((g,r,a);(c,w_1,w_2,\dots,w_r))$ to represent the unpointed gauge group of a Real bundle of class $(c,w_1,w_2,\dots,w_r)$ over a Real surface of type $(g,r,a)$;
\item $\g^*((g,r,a);(c,w_1,w_2,\dots,w_r))$ to represent the single-pointed gauge group of the Real bundle as above;
\item $\g^{*(r+a)}((g,r,a);(c,w_1,w_2,\dots,w_r))$ to represent the $(r+a)$-pointed gauge group of the Real bundle as above.
\end{itemize}

We first present the results relating to when gauge groups of different Real bundles have the same homotopy type. For $(r+a)$-pointed gauge groups this is always the case.
\begin{prop}\label{proprealpointedpathequal}
 Let $(P,\tilde\sigma)$ and  $(P',\sigma ')$ be Real principal $U(n)$-bundles over a Real surface $(X,\sigma)$ of arbitrary type $(g,r,a)$, then there is a homotopy equivalence
 \[\hyperref[proofofproprealpointedpathequal]{B\g^{*(r+a)}(P,\tilde\sigma)\simeq B\g^{*(r+a)}(P',\sigma')}.\]
 \end{prop}
 However, this is not necessarily the case for the single-pointed and unpointed gauge groups, although we do have the following results.
 \begin{prop}\label{propsinglepointedcomponent}
 For any $c, c',w_1, w_1'$ there is a homotopy equivalence
 \[\hyperref[proofofpropsinglepointedcomponent]{B\g^{*}((g,r,a);(c,w_1,w_2,\dots,w_r))\simeq B\g^{*}((g,r,a);(c',w_1',w_2,\dots,w_r))}.\]
 \end{prop}
 
  \begin{prop}\label{propunpointedimmediatefromnonequiv}
 Let the following be classifying spaces of rank $n$ gauge groups. Then there are isomorphisms of gauge groups
 \[\hyperref[proofofprop19]{\g((g,r,a);(c,w_1,w_2,\dots,w_r)) \cong \g((g,r,a);(c+2n,w_1,w_2,\dots,w_r))}.\]
 \end{prop}
 
 \begin{prop}\label{propunpointedcomponentequiv}
 Let $n$ be odd then there are isomorphisms of rank $n$ gauge groups
 
 \begin{enumerate}
 \item \hyperref[proofofpropunpointedcomponentequiv]{$\g((g,r,a);(c,w_1,w_2,\dots,w_r)) \cong \g((g,r,a);(c, \sum_{i=1}^r w_i,0,\dots,0))$};
 \item \hyperref[proofofpropunpointedcomponentequiv]{$\g^{*}((g,r,a);(c,w_1,w_2,\dots,w_r)) \cong \g^{*}((g,r,a);(c, \sum_{i=1}^r w_i,0,\dots,0))$}.
 \end{enumerate} 
 
 \end{prop}
It would be better to provide stronger statements of Propositions \ref{proprealpointedpathequal} and \ref{propsinglepointedcomponent}, such as in the form of the isomorphisms of Propositions \ref{propunpointedimmediatefromnonequiv} and \ref{propunpointedcomponentequiv}.  Indeed, the proofs of the latter invoke a conceptually simple argument and it may be the case that Propositions \ref{proprealpointedpathequal} and \ref{propsinglepointedcomponent} can be given stronger statements using a similar approach.

We state homotopy decompositions for $(r+a)$-pointed gauge group in the form of Theorem \ref{thma}.

\begin{theorem}\label{thma} Let $(P,\tilde\sigma)$ be of arbitrary class then there are integral homotopy decompositions
\[\begin{array}{|c||c|}\hline 
\mbox{ Type }& \mbox{ Decompositions for } \g^{*(r+a)}(P,\tilde\sigma)  \\
\hline\hline 
 (g,0,1) \mbox{ for } g \mbox{ even }  & \hyperref[thm:generalsplitting]{\g^*((0,0,1);0)\times \underset{g}{\prod} \Omega U(n)} \\ \hline
 (g,0,1) \mbox{ for } g \mbox{ odd } &  \hyperref[thm:generalsplitting]{\g^*((1,0,1);0)\times \underset{g-1}{\prod} \Omega U(n)} \\ \hline
 (g,r,0) &  \hyperref[thmsplittingtype1fixedcircles]{\Omega^2(U(n)/O(n))}\times \underset{\hyperref[thm:generalsplitting]{(g-r+1)}+ \hyperref[thmsplittingtype1fixedcircles]{(r-1)}}{\prod} \Omega U(n) \times \hyperref[thmsplittingtype1fixedcircles]{\underset{r-1}{\prod} \Omega O(n)} \\ \hline
 \begin{gathered}\begin{array}{c}
(g,r,1) \\[-9pt]
g-r \mbox{ even}
\end{array}\end{gathered} &  \hyperref[thm111pointed]{\g^*((1,1,1);(0,0))}\times \underset{\hyperref[thm:generalsplitting]{(g-r)}+\hyperref[thmtype2mainsplitting]{(r-1)}+\hyperref[thm111pointed]{1}}{\prod} \Omega U(n) \times \hyperref[thmtype2mainsplitting]{\underset{r-1}{\prod}\Omega O(n)}  \\ \hline
\begin{gathered}\begin{array}{c}
(g,r,1) \\[-9pt]
g-r \mbox{ odd}
\end{array}\end{gathered} &  \hyperref[thm211111integral]{\g^*((1,1,1);(0,0))}\times \underset{\hyperref[thm:generalsplitting]{(g-r-1)}+\hyperref[thmtype2mainsplitting]{(r-1)}+\hyperref[thm211111integral]{2}}{\prod} \Omega U(n) \times \hyperref[thmtype2mainsplitting]{\underset{r-1}{\prod}\Omega O(n)}   \\ \hline
\end{array}\]
\end{theorem}

In the single-pointed case we have to be a little more careful with regards to the class of the underlying Real bundle. For the cases where $\g^{*(r+a)}(P,\tilde\sigma)\neq \g^*(P,\tilde\sigma)$, that is when $r+a>1$, we have the results in Theorem \ref{thmb}.

\setlength{\arraycolsep}{5pt}
\begin{theorem}\label{thmb} Let $n$ be odd or let $(P,\tilde\sigma)$ be of class $(c,w_1,0,\dots, 0)$. Let $r+a>1$ then there are integral homotopy decompositions  \small
\[\begin{array}{|c || c |}\hline
\mbox{ Type }& \mbox{ Decompositions for } \g^*(P,\tilde\sigma) \\
\hline\hline 
 (g,r,0) &  \hyperref[thmsplittingtype1fixedcircles]{\Omega^2(U(n)/O(n))}\times \hyperref[thm:generalsplitting]{\underset{g-r+1}{\prod} \Omega U(n)} \times \hyperref[thmsplittingtype1fixedcircles]{\underset{r-1}{\prod} \Omega O(n)} \times  \hyperref[thmsplittingtype1fixedcircles]{\underset{r-1}{\prod} \Omega (U(n)/O(n))} \\
 \hline
 \begin{gathered}\begin{array}{c}
(g,r,1) \\[-7pt]
g-r \mbox{ even}
\end{array}\end{gathered}  &  \hyperref[thmtype2mainsplitting]{\g^*((1,1,1);(0,0))}\times \hyperref[thm:generalsplitting]{\underset{g-r}{\prod} \Omega U(n)} \times \hyperref[thmtype2mainsplitting]{\underset{r-1}{\prod}\Omega O(n)} \times \hyperref[thmtype2mainsplitting]{\underset{r-1}{\prod} \Omega (U(n)/O(n))}  \\ \hline
 \begin{gathered}\begin{array}{c}
(g,r,1) \\[-7pt]
g-r \mbox{ odd}
\end{array}\end{gathered}  &  \hyperref[thm211111integral]{\g^*((1,1,1);(0,0))}\times \underset{\hyperref[thm:generalsplitting]{(g-r-1)}+\hyperref[thm211111integral]{1}}{\prod} \Omega U(n) \times \hyperref[thmtype2mainsplitting]{\underset{r-1}{\prod}\Omega O(n)} \times \hyperref[thmtype2mainsplitting]{\underset{r-1}{\prod} \Omega (U(n)/O(n))} \\  \hline

\end{array}
\]
\end{theorem}

\setlength{\arraycolsep}{6pt}

The remaining cases seem to integrally indecomposable, however we will obtain the following localised homotopy decompositions for odd rank gauge groups. 

\begin{theorem} \label{thmc}
Let $p\neq 2 $ be prime and let $n$ be odd, then there are the following $p$-local homotopy equivalences
\begin{enumerate}
\item $\hyperref[thmsphereantipointsplitting]{\g^*((0,0,1);c)\simeq_p \Omega^2 (U(n)/O(n))\times \Omega (U(n)/O(n))}$;
\item $\hyperref[thmtorusantipointsplitting]{\g^*((1,0,1);c) \simeq_p \Omega^2 (U(n)/O(n))\times \Omega (U(n)/O(n)) \times \Omega U(n)}$;
\item $\hyperref[thm111pointed]{\g^*((1,1,1);(c,w_1))\simeq_p \Omega^2 (U(n)/O(n)) \times \Omega (U(n)/O(n))\times \Omega O(n)}$.
\end{enumerate}
\end{theorem}
This result relies upon a self map of $U(n)/O(n)$ as studied in \cite{Harrisonthehomotopygroupsoftheclassicalgroups}, which is a $p$-local homotopy equivalence if and only if $n$ is odd.  Hence, it seems to be difficult to provide such satisfactory decompositions in the even rank case.

We move on to some integral homotopy decompositions for unpointed gauge groups.  The reader is invited to compare the tables of Theorem \ref{thmd} and Theorem \ref{thmb}.

\begin{theorem}\label{thmd} 
 Let $(P,\tin)$ be of class $(c,w_1,w_2,\dots,w_r)$ then there are integral homotopy decompositions
\begin{enumerate}\item
$\begin{array}{|c || c |}\hline 
\mbox{ Type } & \mbox{ Decompositions for } \g(P,\tin)  \\

\hline\hline 
 (g,r,0) &  \hyperref[propintegralunpointed]{ \g((r-1,r,0);(c,w_1,\dots,w_r))\times \underset{g-r+1}{\prod} \Omega U(n)} \\ \hline
 \begin{gathered}\begin{array}{c}
(g,r,1) \\[-9pt]
g-r \mbox{ even}
\end{array}\end{gathered}  &   \hyperref[propintegralunpointed]{\g((r,r,1);(c,w_1,\dots,w_r))\times \underset{g-r}{\prod} \Omega U(n)} \\ \hline
 \begin{gathered}\begin{array}{c}
(g,r,1) \\[-9pt]
g-r \mbox{ odd}
\end{array}\end{gathered}  &  \hyperref[propintegralunpointed]{\g((r+1,r,1);(c,w_1,\dots,w_r))\times \underset{g-r-1}{\prod} \Omega U(n) }\\  \hline
\end{array}
$
\vspace{2pt}
\item $ \hyperref[proofoftheoremd3]{\g((2,1,1);(c,w_1))\simeq \g((1,1,1);(c,w_1))\times \Omega U(n)}. $
\end{enumerate} 
Further, for $r\geq1$ and when $(P,\tilde\sigma)$ is of class $(c,w_1,0,\dots,0)$ or $n$ is odd, there are integral homotopy decompositions
\begin{enumerate}
\item[(3)]
$\begin{array}{|c || c |}\hline 
\mbox{ Type } & \mbox{ Decompositions for } \g(P,\tilde\sigma)  \\

\hline\hline 
 (r-1,r,0) &  \hyperref[proofoftheoremd3]{\g((0,1,0);(c,\Sigma w_i))\times \underset{r-1}{\prod} \Omega O(n) \times \underset{r-1}{\prod} \Omega (U(n)/O(n))}  \\ \hline
(r,r,1) &  \hyperref[proofoftheoremd3]{\g((1,1,1);(c,\Sigma w_i))\times \underset{r-1}{\prod} \Omega O(n) \times \underset{r-1}{\prod}\Omega (U(n)/O(n)) } \\ \hline
(r+1,r,1) &  \hyperref[proofoftheoremd3]{\g((2,1,1);(c,\Sigma w_i))\times \underset{r-1}{\prod} \Omega O(n) \times \underset{r-1}{\prod}\Omega (U(n)/O(n)) } \\ \hline

\end{array}
$
\end{enumerate}
\end{theorem}

The remaining unfamiliar spaces in Theorem \ref{thmd} seem to be integrally indecomposable, however localising at particular primes permits further decompositions.
\begin{theorem}\label{thme} Let $n$ be a positive integer and let $p$ be a prime with $p\nmid n$.
 \begin{enumerate}
  \item Let the following be gauge groups of rank $n$ then there are $p$-local homotopy equivalences 
  \begin{enumerate}
   \item $\hyperref[proofofthmea]{\g((g,1,a);(c,0))\simeq_p O(n)\times \g^*((g,1,a);(c,0))}$;
   \end{enumerate}
  further if $p\neq 2$ and $n$ is odd, then there are $p$-local homotopy equivalences
   \begin{enumerate}
   \item[(b)] $\hyperref[proofoftheoreme34]{\g((0,0,1);c)\simeq_p SO(n)\times \Omega^2 (U(n)/SO(n))}$;
   \item[(c)] $\hyperref[proofoftheoreme56]{\g((1,0,1);c)\simeq_p SO(n)\times \Omega^2 (U(n)/SO(n)) \times \Omega U(n)} $.
   
  \end{enumerate}

  \item Let the following be gauge groups of rank $p$ then there are $p$-local homotopy equivalences
  \begin{enumerate}
  \item $\hyperref[proofofthmea]{\g((g,1,a);(c,0))\simeq_p O(p)\times \g^*((g,1,a);(c,0))}$;
  \end{enumerate}
  further if $p\neq 2$, then there are $p$-local homotopy equivalences
  \begin{enumerate}
  \item[(b)] $\hyperref[proofoftheoreme34]{\g((0,0,1);c)\simeq_p SO(p)\times \Omega^2 (U(p)/SO(p))}$;
  \item[(c)] $\hyperref[proofoftheoreme56]{\g((1,0,1);c)\simeq_p SO(p)\times \Omega^2 (U(p)/SO(p)) \times \Omega U(p)} $.
  \end{enumerate}
  
 \end{enumerate}
\end{theorem}

\subsection{Main Results for Quaternionic Bundles}\label{sectionmainresultsforquatbundles}
To distinguish the notation of Quaternionic gauge groups from the Real case we will use a subscript $Q$, for example $\gq(P,\tin)$. Further, to ease notation we will sometimes use the following
\begin{itemize}
\item $\gq((g,r,a);c)$ to represent the unpointed gauge group of a Quaternionic bundle of class $c$ over a Real surface of type $(g,r,a)$;
\item $\gq^*((g,r,a);c)$ to represent the single-pointed gauge group of the Quaternionic bundle as above;
\item $\gq^{*(r+a)}((g,r,a);c)$ to represent the $(r+a)$-pointed gauge group of the Quaternionic bundle as above.
\end{itemize}
We present results in the same order as we did in the Real case. In the Quaternionic case, the homotopy types of the pointed and $(r+a)$-pointed gauge groups are independent of the class of the bundle.
  \begin{prop}\label{propquatcomponentsequivalence}
   Let $(X,\sigma)$ be a Real surface of fixed type $(g,r,a)$. Let $(P,\tilde \sigma)$ and $(P',\sigma')$ be Quaternionic principal $U(2n)$-bundles over $(X,\sigma)$, then there are homotopy equivalences
   \begin{enumerate}
    \item $B\gq^*(P,\tilde \sigma)\simeq B\gq^*(P', \sigma')$;
    \item $B\gq^{*(r+a)}(P,\tilde \sigma)\simeq B\gq^{*(r+a)}(P', \sigma')$.
   \end{enumerate}
  \end{prop}
For the unpointed case, we have an analogue of Proposition \ref{propunpointedimmediatefromnonequiv}.
 \begin{prop}\label{propquatunpointedimmediate}
  Let $(X,\sigma)$ be a Real surface of fixed type $(g,r,a)$ and let the following be gauge groups of Quaternionic bundles of rank $2n$. Then for any even integer $c$, there is an isomorphism of topological groups
  \[ \gq((g,r,a);c))\cong \gq((g,r,a);c+4n).\]
 \end{prop}
 
 We now present homotopy decompositions for pointed gauge groups in the Quaternionic case. The reader is invited to compare the following results to their Real analogues.

\begin{theorem}\label{thmqa} Let $(P,\tilde\sigma)$ be a Quaternionic principal $U(2n)$-bundle of class $c$ then there are integral homotopy decompositions
\[
\begin{array}{|c || c |}\hline 
\mbox{ Type } & \mbox{ Decompositions for } \gq^{*(r+a)}(P,\tilde\sigma)  \\

\hline\hline 
 (g,0,1) \mbox{ for } g \mbox{ even}  & \gq^*((0,0,1);0)\times \underset{g}{\prod} \Omega U(2n) \\ \hline
 (g,0,1) \mbox{ for } g \mbox{ odd} &  \gq^*((1,0,1);0)\times \underset{g-1}{\prod} \Omega U(2n) \\ \hline
 (g,r,0) &  \Omega^2(U(2n)/Sp(n))\times \underset{g}{\prod} \Omega U(2n) \times \underset{r-1}{\prod} \Omega Sp(n) \\ \hline
(g,r,1) \mbox{ for } g-r \mbox{ even} &  \gq^*((1,1,1);0)\times \underset{g}{\prod} \Omega U(2n) \times \underset{r-1}{\prod}\Omega Sp(n)  \\ \hline
(g,r,1) \mbox{ for } g-r \mbox{ odd} &  \gq^*((1,1,1);0)\times \underset{g}{\prod} \Omega U(2n) \times \underset{r-1}{\prod}\Omega Sp(n)  \\ \hline

\end{array}
\]
\end{theorem}

For the cases where $\gq^{*(r+a)}(P,\tilde\sigma)\neq \gq^*(P,\tilde\sigma)$, that is when $r+a>1$, we have the results in Theorem \ref{thmqb}.
\setlength{\arraycolsep}{5pt}
\begin{theorem}\label{thmqb} For $(P,\tin)$ of arbitrary class $c$, there are integral homotopy decompositions\small
\[\begin{array}{|c || c |}\hline 
\mbox{ Type }& \mbox{ Decompositions for }\gq^*(P,\tilde\sigma)  \\
\hline\hline 
 (g,r,0) &  \Omega^2(U(2n)/Sp(n))\times \underset{g-r+1}{\prod} \Omega U(2n) \times \underset{r-1}{\prod} \Omega Sp(n) \times  \underset{r-1}{\prod} \Omega (U(2n)/Sp(n)) \\ \hline
 \begin{gathered}\begin{array}{c}
(g,r,1) \\[-7pt]
g-r \mbox{ even}
\end{array}\end{gathered}&  \gq^*((1,1,1);0)\times \underset{g-r}{\prod} \Omega U(2n) \times \underset{r-1}{\prod}\Omega Sp(n) \times \underset{r-1}{\prod} \Omega( U(2n)/Sp(n)) \\ \hline
 \begin{gathered}\begin{array}{c}
(g,r,1) \\[-7pt]
g-r \mbox{ odd}
\end{array}\end{gathered} &  \gq^*((1,1,1);0)\times \underset{g-r}{\prod} \Omega U(2n) \times \underset{r-1}{\prod}\Omega Sp(n) \times \underset{r-1}{\prod} \Omega (U(2n)/Sp(n))    \\ \hline

\end{array}
\]
\end{theorem}
\setlength{\arraycolsep}{6pt}

Again, the remaining cases seem to integrally indecomposable, however we will obtain the following localised decompositions.

\begin{theorem} \label{thmqc}
Let $p\neq 2 $ be prime, then there are $p$-local homotopy equivalences
\begin{enumerate}
\item $\gq^*((0,0,1);0)\simeq_p \Omega^2 (U(2n)/Sp(n))\times \Omega (U(2n)/Sp(n)); $
\item $\gq^*((1,0,1);0) \simeq_p \Omega^2 (U(2n)/Sp(n))\times \Omega (U(2n)/Sp(n)) \times \Omega U(2n)$;
\item $\gq^*((1,1,1);0)\simeq_p \Omega^2 (U(2n)/Sp(n)) \times \Omega (U(2n)/Sp(n))\times \Omega Sp(n) $.
\end{enumerate}
\end{theorem}

We now present homotopy decompositions for the unpointed case. 

\begin{theorem}\label{thmqd} 
For $(P,\tin)$ of arbitrary class $c$, there are integral homotopy decompositions
\[
\begin{array}{|c || c |}\hline 
\mbox{ Type } & \mbox{ Decompositions for } \gq(P,\tilde\sigma)  \\

\hline\hline 

 \begin{gathered}\begin{array}{c}
(g,0,1) \\[-9pt]
g \mbox{ even}
\end{array}\end{gathered}&  \gq((0,0,1);c)\times \underset{g}{\prod} \Omega U(n) \\ \hline
 \begin{gathered}\begin{array}{c}
(g,0,1) \\[-9pt]
g \mbox{ odd}
\end{array}\end{gathered}&  \gq((1,0,1);c)\times \underset{g-1}{\prod} \Omega U(n) \\ \hline
  (g,r,0) &  \gq((0,1,0);c)\times \underset{r-1}{\prod} \Omega Sp(n) \times \underset{r-1}{\prod} \Omega (U(2n)/Sp(n)) \times \underset{g-r+1}{\prod} \Omega U(n) \\ \hline
(g,r,1) &  \gq((1,1,1);c)\times \underset{r-1}{\prod} \Omega Sp(n) \times \underset{r-1}{\prod}\Omega (U(2n)/Sp(n)) \times \underset{g-r}{\prod} \Omega U(n) \\ \hline

\end{array}
\]
\end{theorem}

The remaining unfamiliar spaces in Theorem \ref{thmqd}
seem to be integrally fundamental, however localising at a particular prime permits further decompositions.
\begin{theorem}\label{thmqe}
 Let $n$ be a positive integer and let $p$ be a prime such that $p\nmid 2n$.  Let the following be gauge groups of a Quaternionic bundle of rank $2n$, then there are $p$-local homotopy equivalences
\begin{enumerate}
  \item $\gq((g,1,a);c)\simeq_p Sp(n)\times B\gq^*((g,1,a);c)$;  
  \item ${\gq((0,0,1);c)\simeq_p Sp(n)\times \Omega^2 \big(U(2n)/Sp(n)\big)}$;
  \item ${\gq((1,0,1);c)\simeq_p Sp(n)\times \Omega^2 \big(U(2n)/Sp(n)\big) \times \Omega U(2n)} $.
  \end{enumerate}
 
\end{theorem}

 \renewcommand{\arraystretch}{1.0}

\section{Proofs of Statements}\label{chapterresults}
For the sake of clarity, we focus on the proofs of statements in the Real case, and then we elaborate on some of the details in the Quaternionic case in Section \ref{sectionthequaternioniccase}.  We look to decompose the gauge groups by studying an equivariant mapping space as provided in \cite{bairdmodulispace}. 

Throughout our analysis, we think of Real surfaces as $\mathbb{Z}_2$-spaces.  For $\mathbb{Z}_2$-spaces $Y$ and $Z$, let $\Mape(Y,Z)$ denote the space of $\mathbb{Z}_2$-maps from $Y$ to $Z$.  We note that the fixed points of $Y$ must be mapped to the fixed points of $Z$. If $Y$ and $Z$ are pointed, we denote a pointed version of this mapping space by $\Mappe(Y,Z)$. Further, recall the `basepoints' $*_i$ of $(X,\sigma)$ just before Definition \ref{defragg}.  Let \[A:=\amalg_{i=1}^{r+1} *_i \amalg \bin(*_{r+1})\] and let $\Map_{\mathbb{Z}_2}^{*(r+a)}(X, Z)$ denote the subspace of $\Map_{\mathbb{Z}_2}(X, Z)$ whose elements send $A$ to $*_{Z}$\footnote{Of course, it may be necessary to assume that $*_Z$ is fixed by the $\mathbb{Z}_2$-action.}.  Let $\overline{X}$ denote the cofibre of  $A \hookrightarrow X$ and notice that there is a homeomorphism
\[\Map^{*{(r+a)}}_{\mathbb{Z}_2} (X, Z)\cong \Mappe(\overline{X}, Z).\]
A universal Real principal $U(n)$-bundle is given by \[(EU(n),\tilde\varsigma)\rightarrow (BU(n),\varsigma)\]
where $\varsigma$ is induced by complex conjugation and hence $BU(n)^\varsigma=BO(n)$.  Using this $\mathbb{Z}_2$-structure, \cite{bairdmodulispace} provides the following theorem.

\begin{theorem}[Baird]\label{thmbaird}
There are homotopy equivalences
\begin{enumerate}
\item $B\g(P,\tilde\sigma)\simeq \Mape(X,BU(n);P)$;
\item $B\g^*(P,\tilde\sigma)\simeq \Mappe(X,BU(n);P)$;
\item $B\g^{*(r+a)}(P,\tilde\sigma)\simeq \Map^{*{(r+a)}}_{\mathbb{Z}_2} (X, BU(n);P)\cong \Mappe(\overline{X}, BU(n);P)$;
\end{enumerate}
where on the right hand side we pick the path component of $\Mape(X,BU(n))$ that induces $(P,\tilde\sigma)$. \qed
\end{theorem}

The following lemma can be shown by adapting the proof in the non-equivariant case. We will frequently require this lemma throughout the paper.
\begin{lemma}\label{lemma:adjointeqloopsus}
  Let $Y$ and $Z$ be $\mathbb{Z}_2$-spaces with basepoints fixed by the action, and with $Y$ locally compact Hausdorff. Then there are equivalences
  \begin{enumerate}
  \item $\Omega \Map_{\mathbb{Z}_2}^{*}(X,Y)\cong \Map_{\mathbb{Z}_2}^{*}(\Sigma X, Y)$;
  \item $\Map_{\mathbb{Z}_2}^{*}(X,\Omega Y)\cong \Map_{\mathbb{Z}_2}^{*}(\Sigma X, Y)$. \qed
  \end{enumerate}
  
\end{lemma}

Throughout this section, there are a number of $\mathbb{Z}_2$-spaces that will often appear; here we provide a dictionary:

\begin{itemize}
\item $(X, \id)$ - any space $X$ with the trivial involution;
\item $(X \vee X, \sw)$ - the wedge $X\vee X$ equipped with the involution that swaps the factors;
\item $(S^n, -\id)$ - the sphere $S^n$ equipped with the antipodal involution;
\item $(S^n, \he)$ - the sphere $S^n$ equipped with the involution that reflects along the equator.
\end{itemize}

\subsection{Real Surfaces as \texorpdfstring{$\mathbb{Z}_2$}{Ztwo}-complexes}\label{sectionkleinsurfacesasz2complexes}
 In order to provide homotopy decompositions for the gauge groups, it will prove useful to provide a $\mathbb{Z}_2$ $CW$-complex structure for Real surfaces.  The following is essentially a restatement of the structures provided in \cite{biswasstablevectorbundles}. We let $\Sigma_{p,q}$ denote a Riemann surface of genus $p$ with $q$ open discs removed.
 \vspace{2pt}
 
\noindent\textbf{Type $(g,0,1)$.}
   We first study the case where $g$ is even.  We can think of $X$ as two copies of $\Sigma_{g/2,1}$ glued  along their boundary components; each a copy of $S^1$.  The involution restricted to $S^1$ is the antipodal map and extends to swap the two copies of $\Sigma_{g/2,1}$.
 
 We give a $CW$-structure of $X$ as follows, let $X^0$ be $2$ zero-cells; $*$ and $\sigma(*)$.  There are $2g+2$ one-cells 
 \begin{gather*}
\alpha_1,\dots,\alpha_{g/2},\beta_1,\dots,\beta_{g/2}, \gamma \mbox{ and}\\
\sigma(\alpha_1),\dots,\sigma(\alpha_{g/2}),\sigma(\beta_1),\dots,\sigma(\beta_{g/2}), \sigma(\gamma).
\end{gather*}
 The boundaries of $\alpha_i, \beta_i$ are glued to $*$ and the boundaries of $\sigma(\alpha_i),\sigma(\beta_i)$ are glued to $\bin(*)$.  One end of $\gamma$ is glued to $*$ and the other to $\sigma(*)$, whilst the same is done for $\sigma(\gamma)$ with the opposite orientation.  There are $2$ two-cells glued on, one with attaching map
\[\alpha_1\beta_1\alpha_1^{-1}\beta_1^{-1}\cdots\alpha_{g/2}\beta_{g/2}\alpha_{g/2}^{-1}\beta_{g/2}^{-1}\gamma\sigma(\gamma)\]
and the other with the same attaching map but with $\alpha_i, \beta_i$ replaced with $\sigma(\alpha_i),\sigma({\beta_i})$ and $\gamma\sigma(\gamma)$ replaced with $\sigma(\gamma)\gamma$. 

As the notation suggests, the involution swaps cells that differ by $\sigma$.  In particular, this is a $\sigma$-equivariant $CW$-structure and hence descends to a $CW$-structure of $X/\sigma$. 

Now assume that $g$ is odd and let $g'=(g-1)$.  We see that $X$ can be thought of as two copies of $\Sigma_{g'/2,2}$ glued along their boundaries; two copies of $ S^1$ in $X$. The involution swaps these copies of $S^1$ but reverses orientations, and it extends to $X$ to swap the two copies of $\Sigma_{g'/2,2}$. 

There are 2 zero-cells, $*$ and $\sigma(*)$ and $2g$ one-cells 
\begin{gather*}
\alpha_1,\dots,\alpha_{g'/2},\beta_1,\dots,\beta_{g'/2}, \gamma,\delta \mbox{ and}\\
\sigma(\alpha_1),\dots,\sigma(\alpha_{g'/2}),\sigma(\beta_1),\dots,\sigma(\beta_{g'/2}), \sigma(\gamma),\sigma(\delta)
\end{gather*}
where $\alpha_i,\beta_i,\sigma(\alpha_i),\sigma(\beta_i),\gamma,\sigma(\gamma)$ are glued as before but the boundary of $\delta$ is glued to $*$ and $\sigma(\delta)$ to $\sigma(*)$.  Now there are $2$ two-cells, one with boundary map
\[\alpha_1\beta_1\alpha_1^{-1}\beta_1^{-1}\cdots\alpha_{g'/2}\beta_{g'/2}\alpha_{g'/2}^{-1}\beta_{g'/2}^{-1}\delta\gamma\sigma(\delta)\gamma^{-1}\]
and the other glued equivariantly.  The cells $\delta$ and $\sigma(\delta)$ correspond to the copies of $S^1$ above and here $\gamma$ is a cell joining these copies of $S^1$.  
 \vspace{2pt}
 
\noindent \textbf{Type $(g,r,0)$.}  Let the involution fix $r$ circles and let $g'=(g-r+1)/2$, then $X/\sigma$ is a $\Sigma_{g',r}$ and $X$ can be thought of as two copies of $\Sigma_{g',r}$ glued along the $r$ boundary components.  

In this case, the basepoint is preserved under $\sigma$, however $X^0$ is given $r$ zero-cells; one for each fixed component.  The one cells are then 
\begin{gather*}
\alpha_1,\dots,\alpha_{g'},\beta_1,\dots,\beta_{g'}, \gamma_2,\dots,\gamma_r,\delta_1,\dots,\delta_r \mbox{ and}\\
\sigma(\alpha_1),\dots,\sigma(\alpha_{g'}),\sigma(\beta_1),\dots,\sigma(\beta_{g'}), \sigma(\gamma_2),\dots,\sigma(\gamma_r)
\end{gather*}
where $\alpha_i,\beta_i$ are as before and $\gamma_i$ joins the basepoint to the $i$-th fixed component which is represented by $\delta_i$.  One of the $2$ two-cells has attaching map
\[\alpha_1\beta_1\alpha_1^{-1}\beta_1^{-1}\cdots\alpha_{g'}\beta_{g'}\alpha_{g'}^{-1}\beta_{g'}^{-1}\delta_1\gamma_2\delta_2\gamma_2^{-1}\cdots\gamma_r\delta_r\gamma_r^{-1}\]
and we again define the other one equivariantly.
 \vspace{2pt}
 
\noindent \textbf{Type $(g,r,1)$ for $r>0$.} 
Let the involution fix $r$ circles. We first consider the case where $g\equiv r \mod 2$.  Let $g'=(g-r)/2$, then $X$ can be thought of as two copies of $\Sigma_{g',r+1}$ glued along the boundary components.  The involution fixes the first $r$ of these components whilst restricting to the antipodal map on the extra copy of $S^1$.

Now, $X^0$ is given $(r+2)$ zero-cells $*_i$; one for each fixed component and two for the extra $S^1$.  The one cells are then 
\begin{align*}
\alpha_1,&\dots,\alpha_{g'},\beta_1,\dots,\beta_{g'}, \gamma_2,\dots,\gamma_{r+1},\delta_1,\dots,\delta_{r},\delta \mbox{ }\mbox{and}\\
\sigma(\alpha_1),&\dots,\sigma(\alpha_{g'}),\sigma(\beta_1),\dots,\sigma(\beta_{g'}), \sigma(\gamma_2),\dots,\sigma(\gamma_{r+1}),\sigma(\delta)
\end{align*}
where $\alpha_i,\beta_i$ are as before and $\gamma_i$ joins the basepoint to the $i$-th boundary circle. Each fixed component is represented by $\delta_i$ and $\delta$ joins $*_{r+1}$ to $*_{r+2}$ and therefore $\delta\sigma(\delta)$ represents the extra copy of $S^1$.  One of the $2$ two-cells has attaching map
\[\alpha_1\beta_1\alpha_1^{-1}\beta_1^{-1}\cdots\alpha_{g'}\beta_{g'}\alpha_{g'}^{-1}\beta_{g'}^{-1}\delta_1\gamma_2\delta_2\gamma_2^{-1}\cdots\gamma_r\delta_r\gamma_r^{-1}\gamma_{r+1}\delta\sigma(\delta)\gamma^{-1}_{r+1}\]
and we again define the other one equivariantly.

For the case where, $g\equiv r+1 \mod 2$, we let $g'=(g-r-1)/2$.  Now $X$ can be thought of as two copies $\Sigma_{g',r+2}$ glued along the boundary components.  Again, the involution fixes $r$ of these components, whilst swapping the final two copies of $S^1$ but reversing orientation.

Again $X^0$ is given $(r+2)$ zero-cells; one for each fixed component and one for each of the extra two copies of $S^1$.  The one cells are then
\begin{align*}
\alpha_1,&\dots,\alpha_{g'},\beta_1,\dots,\beta_{g'}, \gamma_2,\dots,\gamma_{r+2},\delta_1,\dots,\delta_{r+1} \mbox{ }\mbox{and}\\
\sigma(\alpha_1),&\dots,\sigma(\alpha_{g'}),\sigma(\beta_1),\dots,\sigma(\beta_{g'}), \sigma(\gamma_2),\dots,\sigma(\gamma_{r+2}),\sigma(\delta_{r+1})
\end{align*}
where $\alpha_i,\beta_i$ are as before and $\gamma_i$ joins the basepoint to the $i$-th boundary circle. Each fixed component is represented by $\delta_i$ for $i\leq r$, and $\delta_{r+1}$ and $\sigma(\delta_{r+1})$ represent the extra copies of $S^1$.
One of the $2$ two-cells has attaching map
\[\alpha_1\beta_1\alpha_1^{-1}\beta_1^{-1}\cdots\alpha_{g'}\beta_{g'}\alpha_{g'}^{-1}\beta_{g'}^{-1}\delta_1\gamma_2\delta_2\gamma_2^{-1}\cdots\gamma_{r+1}\delta_{r+1}\gamma^{-1}_{r+1}\gamma_{r+2}\sigma(\delta_{r+1})\gamma^{-1}_{r+2}\]
and we again define the other one equivariantly.

\subsection{Equivalent Components of Mapping Spaces}\label{sectionequivalentcomponentsofmappingsspaces}
In this section we aim to prove Propositions \ref{proprealpointedpathequal}--\ref{propunpointedcomponentequiv}. The proofs are motivated from the analysis of non-equivariant mapping spaces found in \cite{sfunctionspaces}.

 \begin{proof}[Proof of Proposition \ref{proprealpointedpathequal}]\label{proofofproprealpointedpathequal}
  We study the actions of $\pi_2 (BU(n))$ and $\pi_1(BO(n))$ on the components of $\Mappe(\overline{X}, BU(n))$.  In \cite{sfunctionspaces}, an action of $\pi_2(BU(n))$ on $\Map(X,BU(n))$ was defined via
  \begin{equation}\label{eq:pinchaction} X\xrightarrow{\mbox{\footnotesize pinch}} X\vee S^2 \xrightarrow{f\vee \alpha}BU(n)\vee BU(n) \xrightarrow{\mbox{\footnotesize fold}} BU(n)\end{equation}
  with $\alpha \in \pi_2(BU(n))$ and $f\in \Map^*(X,BU(n))$. 
  
  We now consider the equivariant case for $r=0$.  Let $S^1$ be the loop that is pinched under $\overline X\rightarrow \overline X\vee S^2$; similar to the first map in equation (\ref{eq:pinchaction}).  Due to equivariance, we are also forced to pinch the loop $\sigma(S^1)$ producing an extra factor of $S^2$ and the action becomes
    \[ \overline X\xrightarrow{\mbox{\footnotesize pinch}}\overline X\vee S^2\vee \sigma(S^2) \xrightarrow{f\vee \alpha\vee \overline{\alpha}}BU(n)\vee BU(n)\vee BU(n) \xrightarrow{\mbox{\footnotesize fold}} BU(n).\]
    where $\overline \alpha= \varsigma \alpha$.
Since $\sigma$ and $\varsigma$ are both orientation reversing, the action of \[\alpha\in\pi_2(BU(n))\cong \mathbb{Z}\] alters the class $[f]$ by $2\alpha$.  Hence for $2c\in [\overline X,BU(n)]_{\mathbb{Z}_2}\cong 2\mathbb{Z}$, this action gives homotopy equivalences \[
\Mappe(X,BU(n);2c)\simeq \Mappe(X,BU(n);2c+2\alpha).\]
In particular this gives the required homotopy equivalences for the case when $r=0$.

When $r>0$, the path components of $\Mappe(\overline X, BU(n))$ are classified by the tuple
\[(c,w_1,w_2,\dots,w_r)\in \mathbb{Z}\times \prod_{r}\mathbb{Z}_2\]
subject to $c\equiv \sum_{i=1}^r w_i \mod 2$. We wish to construct an action of $\pi_1 (BO(n))$ to alter each $w_i$.  For $\beta \in \pi_1 (BO(n))$, we note that the inclusion of the image of $\beta$ into $BU(n)$ is nullhomotopic, so there is an extension $\beta'\colon D^2 \rightarrow BU(n)$ of $\beta$.  Now, consider $(S^2,\he)$ and denote the fixed equator by $E$, the upper hemisphere by $U$ and the lower hemisphere by $L$.  We can extend $\beta$ to a map $\tilde \beta\colon (S^2,\he) \rightarrow BU(n)$ where \[\tilde \beta \mid_U = \beta' \mbox{ and }\tilde \beta \mid_L = \varsigma\beta'\] and therefore $\tilde \beta \mid_E = \beta$.  Due to the discussion preceding Proposition $4.1$\footnote{We note that Proposition 4.1 in \cite{biswasstablevectorbundles} is stated as Proposition \ref{proprealbundleclass} in this paper.} in \cite{biswasstablevectorbundles}, the extension $\tilde \beta$ can be chosen so that the class $[\tilde\beta] \in \mathbb{Z}\times \mathbb{Z}_2$ is $(0,0)$ if $\beta$ is trivial or $(\
pm 1,1)$ otherwise.

Let $(S^1,\he)\hookrightarrow \overline X$ be an inclusion such that the fixed points of $(S^1,\he)$ are mapped to the $i$-th fixed component $X_i$ of $\overline X$.
As in equation (\ref{eq:pinchaction}) we apply the pinch map to this copy of $(S^1,\he)$ in $\overline X$, and hence produce a factor of $(S^2,\he)$.  Now the action becomes
  
  \begin{equation*} \overline X\xrightarrow{\mbox{\footnotesize pinch}} \overline X\vee (S^2,\he) \xrightarrow{f\vee \tilde\beta}BU(n)\vee BU(n) \xrightarrow{\mbox{\footnotesize fold}} BU(n).\end{equation*}
  For $\tilde \beta$ of class $(\pm 1,1)$, we conclude that this action gives a homotopy equivalence between the components $(c,w_1,w_2,\dots,w_r)$ and $(c\pm 1, w_1,\dots,w_i+1,\dots,w_r)$.
   Combining the actions of $\pi_2 (BU(n))$ and $\pi_1(BO(n))$ gives homotopy equivalences between all the components of $\Mappe(\overline X,BU(n))$.
 \end{proof}
 
 \begin{proof}[Proof of Proposition \ref{propsinglepointedcomponent}] \label{proofofpropsinglepointedcomponent} Recall, from the preamble to Definition \ref{defragg}, that we chose $*_1$ as the basepoint of $(X,\bin)$.  We define actions of $\pi_2 (BU(n))$ and $\pi_1(BO(n))$ on $\Mappe(X,BU(n))$ in a similar fashion to the proof of Proposition \ref{proprealpointedpathequal}, and this obtains the result.  We cannot extend this result as in the $(r+a)$-pointed case due to the `unpointed' fixed circles. 
 \end{proof}
 
 We cannot hope to use the actions of $\pi_1$ and $\pi_2$ on the unpointed mapping space due to the lack of basepoint.  But, by tensoring the bundle $(P,\tin)$ with a Real $U(1)$-bundle, we can provide some equivalences between components.
 
\begin{proof}[Proof of Proposition \ref{propunpointedimmediatefromnonequiv}]\label{proofofprop19}
 Let $\pi \colon (P,\tin)\rightarrow (X,\bin)$ be a Real principal $U(n)$-bundle of class $(c,w_1,w_2,\dots,w_r)$ over a Real surface of type $(g,r,a)$.  The idea will be to tensor $P$ with a Real $U(1)$-bundle $\pi_Q\colon (Q,\tau)\rightarrow (X,\bin)$ of class $(2,0,\dots,0)$. 
 
 Using the inclusion of the centre $U(1)\hookrightarrow U(n)$, there is a $U(1)$-action on $(P,\tin)$. In the principal bundle setting, the tensor of $(P,\tin)$ and $(Q,\tau)$ is the pullback
 \[\xymatrix{ (\Delta^*(P \times_{U(1)} Q), \Delta^*(\tin \times \tau)) \ar[r] \ar[d] & (P \times_{U(1)} Q, \tin \times \tau) \ar[d]^{\tilde \pi} \\
 (X,\bin) \ar[r]^-{\Delta} & (X,\bin)\times (X,\bin).
 }\]
 where $\Delta$ is the diagonal map and $\tilde \pi=\pi\times \pi_Q$. In a similar fashion to the discussion preceding Proposition 4.1 in \cite{biswasstablevectorbundles}, we calculate that the class of the pullback $(\Delta^*(P \times_{U(1)} Q), \Delta^*(\tin \times \tau))$ is $(c+2n,w_1,w_2,\dots,w_r)$.
 
 We then define
 \[\Theta \colon \g(P,\tin)\rightarrow \g(\Delta^*(P \times_{U(1)} Q), \Delta^*(\tin \times \tau))\]
 to be the map that sends $\phi\colon P\rightarrow P$ to $\Delta^*(\phi\times \id)$. Then an inverse to $\Theta$ is defined in the same way as $\Theta$, except that we replace the inclusion $U(1) \hookrightarrow U(n)$ with the conjugate inclusion defined via
 \[ 
 a  \mapsto  \left(\begin{array}{cccc}
 \overline{a} & 0 & \cdots & 0 \\
 0 & \overline a & \cdots & 0 \\
 \vdots &\vdots & \ddots & \vdots \\
 0 & 0 & \cdots & \overline a
 \end{array}\right) \hspace{-7pt}\begin{array}{l} \vphantom{\vdots}\\ \vphantom{\vdots}\\ \vphantom{0} \\ .\end{array}
 \]
 \end{proof}
 
 \begin{proof}[Proof of Proposition \ref{propunpointedcomponentequiv}]\label{proofofpropunpointedcomponentequiv}
 Let $\pi \colon (P,\tin)\rightarrow (X,\bin)$ be a Real principal $U(n)$-bundle of class $(c,w_1,w_2,\dots,w_r)$ over a Real surface of type $(g,r,a)$. The statement is proven using the same method as Proposition \ref{propunpointedimmediatefromnonequiv}, except that we tensor with a Real $U(1)$-bundle $(\tilde Q,\tilde \tau)$ of class $(0,\sum_{i=2}^r w_i,w_2,\dots,w_r)$.  If $n$ is odd, the class of the pullback $(\Delta^*(P \times_{U(1)} \tilde Q), \Delta^*(\tin \times \tilde \tau))$ is then $(c,\sum_{i=1}^r w_i,0,\dots,0)$.  
 
 An isomorphism $\Theta \colon \g(P,\tin)\rightarrow \g(\Delta^*(P \times_{U(1)} \tilde Q), \Delta^*(\tin \times \tilde \tau))$ is then defined in the same way as for Proposition \ref{propunpointedimmediatefromnonequiv}.
 \end{proof}

\subsection{Pointed Gauge Groups}\label{pointedgaugegroups}
 In the following analysis, it will be necessary to distinguish the following types of Real surfaces
 \begin{enumerate}
 \item[\textbf{0.}]  $r=0$ $(\Rightarrow a=1)$ 
 \item[\textbf{1.}]  $r>0$ and $a=0$
 \item[\textbf{2.}]  $r>0$ and $a=1$.
 \end{enumerate}
Generally we will analyse the gauge groups in order of ease; we first analyse the $(r+a)$-pointed gauge group and then the single pointed gauge group.  Our results for the single pointed gauge groups will then be used to analyse the unpointed case.

\subsubsection{Integral Decompositions}\label{subsectionintegraldecompositions}
For the underlying Riemann surface $X$ of a Real surface $(X,\sigma)$, the attaching map $f\colon S^1 \rightarrow \vee_{2g} S^1$ of the top cell is a sum of Whitehead products and hence the suspension $\Sigma f$ is nullhomotopic. 
In the Real surface case, we see Whitehead products appearing in the attaching maps of Section \ref{sectionkleinsurfacesasz2complexes}. Therefore, we still see trivialities appearing in the suspension of these attaching maps and these trivialities will provide a large class of homotopy decompositions. 

We will use the notation as defined in Section \ref{sectionkleinsurfacesasz2complexes} and furthermore we require the following notation in this section. Let $g'$ denote the number of one-cells of $X$ which are of the form $\alpha_i, \beta_i$ in $X$. Explicitly
\[g'=\begin{cases} 
(g-r+1)  &\mbox{ when } a=0; \\
(g-r)  & \mbox{ when } a=1\mbox{ and } g-r \mbox{ even;} \\
(g-r-1)  &\mbox{ when } a=1\mbox{ and } g-r \mbox{ odd.} \\
\end{cases}\]
\begin{prop}\label{propgtrivial}
Let $X_{\alpha\beta}=\vee S^1$ be the $1$-cells $\alpha_i,\bin(\alpha_i), \beta_i, \bin(\beta_i)$ in the decomposition of $(X, \bin)$.  Then the map $\mu$ in the $\mathbb{Z}_2$-cofibration sequence
\[X_{\alpha\beta}\hookrightarrow X \rightarrow X' \xrightarrow{\mu} \Sigma(X_{\alpha\beta})\]
is $\mathbb{Z}_2$-nullhomotopic.
\end{prop}
\begin{proof} We recall that the attaching map of one of the two-cells in a Real surface of type $(g,r,0)$ is
\[\alpha_1\beta_1\alpha_1^{-1}\beta_1^{-1}\cdots\alpha_{g'}\beta_{g'}\alpha_{g'}^{-1}\beta_{g'}^{-1}\delta_1\gamma_2\delta_2\gamma_2^{-1}\cdots\gamma_r\delta_r\gamma_r^{-1}.\] 
The attaching map involving the cells $\alpha_i$ and $\beta_i$ is a sum of Whitehead products.  The idea is to collapse the rest of the cells.

Now in the general case, let $X$ be a type $(g,r,a)$ Real surface, let $\Sigma_{g'/2}$ be a Riemann surface of genus $g'/2$ and let \[s\colon X\rightarrow (\Sigma_{g'/2}\vee \Sigma_{g'/2},\sw)\] be the map that collapses the $1$-skeleton of $X$ other than the cells $\alpha_i,\bin(\alpha_i), \beta_i$ and $\bin(\beta_i)$. 

An example for the map $s$ is illustrated in Figure \ref{figurecollapseexceptalphabeta}. Note that four of the `holes' are undisturbed by $s$; these correspond to the one-cells of the form $\alpha_i, \bin(\alpha_i), \beta_i $ and $\bin(\beta_i)$.
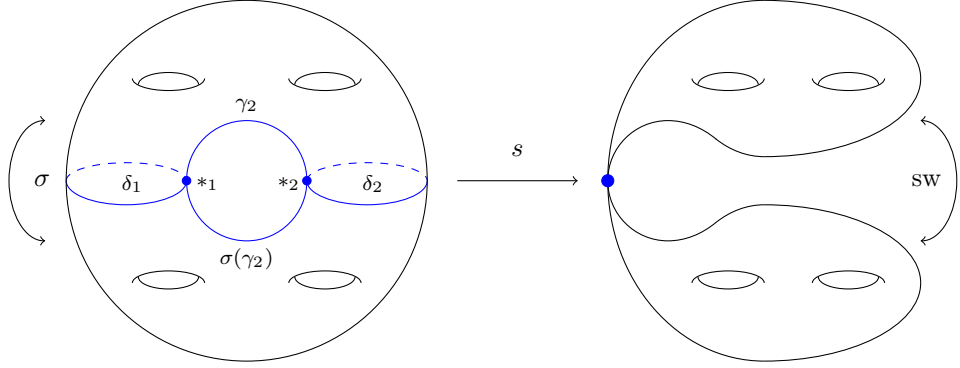
\begin{figure}[ht]  \caption{For a type $(5,2,0)$ Real surface, the map $s$ collapses the one-cells coloured in blue: $\delta_1$, $\delta_2$, $\gamma_2$ and $\bin(\gamma_2)$. } \label{figurecollapseexceptalphabeta}
\centering
\begin{tikzpicture}[scale=0.8]
	
\draw[dashed,color=colour1] (0,0) arc (180:0:1 and 0.3);
	\draw[color=colour1](0,0) arc (180:360:1 and 0.4);
	
\draw[dashed,color=colour1] (4,0) arc (180:0:1 and 0.3);
	\draw[color=colour1] (4,0) arc (180:360:1 and 0.4);
	
\draw[colour1] (3,0) circle [x radius=1cm, y radius=1cm];;	

\draw[black] (3,0) circle [x radius=3cm, y radius=3cm];

\draw[black] (1.2,1.60) arc (180:0:0.5 and 0.2);
	\draw[black] (1.1,1.7) arc (180:360:0.6 and 0.2);
	
\draw[black] (1.2,-1.60) arc (-180:0:0.5 and 0.2);
	\draw[black] (1.1,-1.7) arc (-180:-360:0.6 and 0.2);

\draw[black] (3.8,1.60) arc (180:0:0.5 and 0.2);
	\draw[black] (3.7,1.7) arc (180:360:0.6 and 0.2);
	
\draw[black] (3.8,-1.60) arc (-180:0:0.5 and 0.2);
	\draw[black] (3.7,-1.7) arc (-180:-360:0.6 and 0.2);

\draw(-0.35,1) [<->] arc (90:-90:-0.6 and 1);
\draw (-0.4,0) node {$\sigma$};

\filldraw[colour1] (2,0) circle (0.07);
\filldraw[colour1] (4,0) circle (0.07);

\draw (2.35,-0.05) node {\footnotesize $*_1$};
\draw (3.7,-0.05) node {\footnotesize $*_2$};
\draw (1.1,-0.05) node {\footnotesize $\delta_1$};
\draw (5.1,-0.05) node {\footnotesize $\delta_2$};
\draw (3,1.25) node {\footnotesize $\gamma_2$};
\draw (3,-1.3) node {\footnotesize $\sigma(\gamma_2)$};
\draw[->] (6.5,0)--(8.5,0);
\draw (7.5,0.5) node {$s$};

\draw[black] (11.6,3) arc (90:-90:2.6 and 1.3);
\draw[black] (9,0) arc (180:90:2.6 and 3);
\draw[black] (9,0) arc (180:90:1 and 1);
\draw[black] (10,1) arc (90:60:1.6 and 2.3);
\draw[black] (11.6,0.4) arc (-90:-120:1.61 and 2.21);

\draw[black] (11.6,-0.4) arc (90:-90:2.6 and 1.3);
\draw[black] (9,0) arc (180:270:2.6 and 3);
\draw[black] (9,0) arc (180:270:1 and 1);
\draw[black] (10,-1) arc (-90:-60:1.6 and 2.3);
\draw[black] (11.6,-0.4) arc (90:120:1.61 and 2.21);

\draw[black] (10.5,1.60) arc (180:0:0.5 and 0.2);
	\draw[black] (10.4,1.7) arc (180:360:0.6 and 0.2);
	
\draw[black] (10.5,-1.60) arc (-180:0:0.5 and 0.2);
	\draw[black] (10.4,-1.7) arc (-180:-360:0.6 and 0.2);

\draw[black] (12.5,1.60) arc (180:0:0.5 and 0.2);
	\draw[black] (12.4,1.7) arc (180:360:0.6 and 0.2);
	
\draw[black] (12.5,-1.60) arc (-180:0:0.5 and 0.2);
	\draw[black] (12.4,-1.7) arc (-180:-360:0.6 and 0.2);
	
\draw(14.2,1) [<->] arc (90:-90:0.6 and 1);
\draw (14.3,0) node {$\sw$};

\filldraw[colour1] (9,0) circle (0.1);

\end{tikzpicture}
\end{figure}

There is a commutative diagram
\[\xymatrix{ X_{\alpha\beta} \ar@{=}[d]\ar[r] &X \ar[r] \ar[d]^s &X' \ar[d]^{s'} \ar[r]^-\mu & \Sigma(X_{\alpha\beta}) \ar@{=}[d] \\ X_{\alpha\beta}\ar[r] &  (\Sigma_{g'/2}\vee \Sigma_{g'/2},\sw) \ar[r] &(S^2\vee S^2,\sw) \ar[r]^-{\Sigma f\vee \overline{\Sigma f}} & \Sigma(X_{\alpha\beta})}\]
where the rows are $\mathbb{Z}_2$-cofibration sequences, $s'$ is an induced map on cofibers and $f$ is the attaching map of the Riemann surface $\Sigma_{g'/2}$.  The $\mathbb{Z}_2$-triviality of $\mu$ therefore follows from the triviality of $\Sigma f$.
\end{proof}

We deduce the following theorem which contributes a lot of results to Theorems \ref{thma} and \ref{thmb}.
\begin{theorem}\label{thm:generalsplitting}
 With notation as above, there are homotopy equivalences
 \begin{enumerate}\item $\g^*(P,\tilde\sigma)\simeq \g^*((g-{g'},r,a);(c,w_1,\dots,w_r))\times \prod_{g'} \Omega U(n);$
\item $\g^{(r+a)*}(P,\tilde\sigma)\simeq \g^{(r+a)*}((g-{g'},r,a);(c,w_1,\dots,w_r))\times \prod_{g'} \Omega U(n). $
\end{enumerate}
\end{theorem}

\begin{proof}
 We use the notation of Proposition \ref{propgtrivial} and run through details for part $(1)$. By Theorem \ref{thmbaird}, there is a homotopy fibration sequence
 \[ \g(P,\tin) \rightarrow \Omega \Mappe(X_{\alpha\beta},BU(n)) \xrightarrow{\mu^*} \Mappe (X', BU(n);(c,w_1,\dots w_r))\]
 and, by Lemma \ref{lemma:adjointeqloopsus}, we can see that $\mu^*$ is induced from $\mu$ in Proposition \ref{propgtrivial}. But Proposition \ref{propgtrivial} showed that $\mu$ is nullhomotopic and the result follows.  The proof for $(2)$ is similar.
\end{proof}
We note that for Real surfaces of type $(g,0,1)$, Theorem \ref{thm:generalsplitting} leaves only types $(0,0,1)$ and $(1,0,1)$ to consider.  The gauge groups of these types seem to be integrally indecomposable and so we leave their analysis until later.  

\subsubsection{Case: \texorpdfstring{$r>0$, $a=0$}{r>0,a=0}}
Although we restrict to the case $a=0$, we will see that many of the methods in this section will also transfer to the case when $a=1$.  

Due to Theorem \ref{thm:generalsplitting} we restrict to the case when $(X,\bin)$ is of type $(r-1,r,0)$. For $(P,\tin)$ of class $(0,0,\dots,0)$, we utilise Theorem \ref{thmbaird} and Lemma \ref{lemma:adjointeqloopsus} and obtain the equivalences
  \begin{align*} \g^{*r}(P,\tilde\sigma) & \simeq \Map^*_{\mathbb{Z}_2}(\Sigma (\overline{X}), BU(n));\\
\g^{*}(P,\tilde\sigma) &\simeq \Map^*_{\mathbb{Z}_2}(\Sigma (X), BU(n)).\end{align*}
The aim of this section is to prove Theorems \ref{thma}
and \ref{thmb}
for types $(g,r,0)$ which is restated below.
\begin{theorem}\label{thmsplittingtype1fixedcircles}
Let $(P,\tilde\sigma)$ be a Real bundle of class $(c,w_1,\dots,w_r)$ over a Real surface $(X,\bin)$ of type $(r-1,r,0)$.  Then 
\begin{enumerate}
\item there is a homotopy equivalence \[\g^{*r}(P,\tilde\sigma)\simeq \Omega^2 (U(n)/O(n))\times \prod_{r-1} \Omega O(n) \times \prod_{r-1} \Omega U(n);\]
\item if $w_i=0$ for all $i>1$ or if $n$ is odd then there is a homotopy equivalence \[\g^*(P,\tilde\sigma)\simeq \Omega^2(U(n)/O(n))\times \prod_{r-1} \Omega O(n) \times \prod_{r-1} \Omega(U(n)/O(n)).\]
\end{enumerate}
\end{theorem}
 Recall the $\mathbb{Z}_2$-structure of a type $(g,r,0)$ surface in Section \ref{sectionkleinsurfacesasz2complexes}. In the following $X_\gamma$ will be the sub-complex of the one-cells of $X$ that are denoted by either $\gamma_i$ or $\sigma(\gamma_i)$.
 \begin{prop}\label{propgtrivialfortype1}
 Let $(X,\bin)$ be as above, then in the $\mathbb{Z}_2$-cofibration sequence
 \[X_\gamma \xrightarrow{\iota} X \rightarrow \tilde X \xrightarrow{\mu'} \Sigma(X_\gamma)\]
 there is a left $\mathbb{Z}_2$-homotopy inverse to $\iota$.  In particular $\mu'$ is $\mathbb{Z}_2$-nullhomotopic.
 \end{prop}
 \begin{proof}
    We will use induction on $r$; the number of fixed circles of $X$.  Let $X_r$ denote a Real surface of type $(r-1,r,0)$ and let $(X_r)_\gamma$ be the sub-complex of $X_r$ with one-cells denoted by either $\gamma_i$ or $\sigma(\gamma_i)$.  We aim to define left homotopy inverses $j_r\colon X_r \rightarrow (X_r)_\gamma$ of $\iota$ for each $r$. 

Note that the space $(X_r)_\gamma$ is the wedge $\bigvee_{r-1}(S^1,\he)$ and hence the first non-trivial case is when $r=2$. In this case, one can see that $X_2$ is the product
\[(S^1,\id)\times (S^1,\he).\] We define $j_2$ to be the projection onto the second factor and Figure \ref{figurethemapj2} illustrates this map.

\begin{figure}[ht] \caption{The map $j_2$ projects to the factor $(S^1,\he)$ and $j_2$ factors through $X_2/\delta_1$.} \label{figurethemapj2}
\centering
\begin{tikzpicture}[scale=0.7]
	
\draw[dashed,color=colour1] (0,0) arc (180:0:1 and 0.3);
	\draw[color=colour1](0,0) arc (180:360:1 and 0.4);
	
\draw[dashed,color=black] (4,0) arc (180:0:1 and 0.3);
	\draw[color=black] (4,0) arc (180:360:1 and 0.4);
	
\draw[black] (3,0) circle [x radius=1cm, y radius=1cm];	

\draw[black] (3,0) circle [x radius=3cm, y radius=3cm];

\draw(-0.35,1) [<->] arc (90:-90:-0.6 and 1);
\draw (-0.4,0) node {$\sigma$};

\filldraw[colour1] (2,0) circle (0.07);
\filldraw[black] (4,0) circle (0.07);

\draw (1.1,-0.05) node {\textcolor{colour1}{\footnotesize $\delta_1$}};
\draw (3,1.25) node {\footnotesize $\gamma_1$};
\draw (3,-1.3) node {\footnotesize $\sigma(\gamma_1)$};

\draw (15,1.25) node {\footnotesize $\gamma_1$};
\draw (15,-1.3) node {\footnotesize $\sigma(\gamma_1)$};
\draw (13,0.4) node {\footnotesize $\tilde j$};

\draw[->] (6.5,0)--(7.5,0);
\draw (7,0.4) node {\scriptsize Collapse};
\draw (7,-0.4) node {\scriptsize $\delta_1$};
\draw (12,-2.3) node {\scriptsize $X_2/\delta_1$};
\draw[black] (9,0) circle [x radius=1cm, y radius=1cm];
\draw[black] (10,0) circle [x radius=2cm, y radius=2.3cm];

\draw[dashed,color=black] (10,0) arc (180:0:1 and 0.3);
	\draw[color=black] (10,0) arc (180:360:1 and 0.4);
	
\filldraw[colour1] (8,0) circle (0.07);
\filldraw[black] (10,0) circle (0.07);


\draw[->] (12.5,0)--(13.5,0);

\draw[black] (15,0) circle [x radius=1cm, y radius=1cm];
\filldraw[black] (14,0) circle (0.07);
\filldraw[black] (16,0) circle (0.07);

\draw(16.3,0.8) [<->] arc (90:-90:0.6 and 0.8);
\draw (16.4,0) node {\footnotesize $\he$};

\draw (9.7,3.2) node { $j_2$};

\draw(5.3,2.3) [->] arc (150:10:5 and 2.5);
\draw[->] (12.5,0)--(13.5,0);
\end{tikzpicture}
\end{figure}
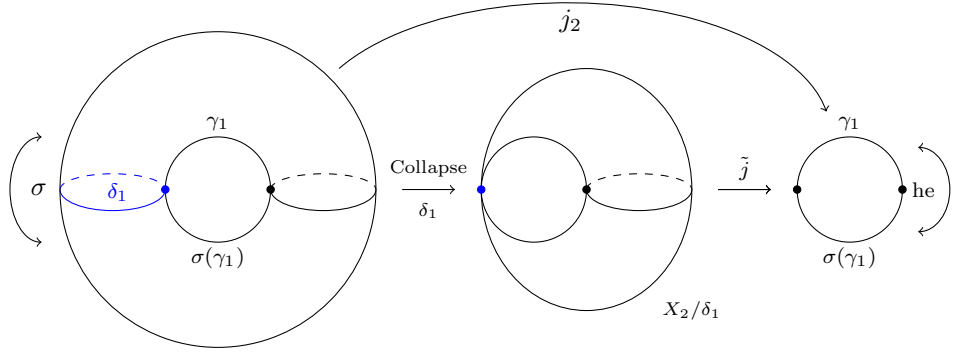
For $r=l$, we assume that $j_l$ exists. For $r=l+1$, we first use a map $j_{l+1}'$ that collapses a copy of $(S^1 \vee S^1, \sw)$ in $X_{l+1}$ such that the image is homeomorphic to $X_{l}\vee X_2/\delta_1$ where $X_2/\delta_1$ is a copy $X_2$ with the $1$-cell $\delta_1$ collapsed.  The map $j_{l+1}'$ is illustrated in Figure \ref{figurejtilde}.

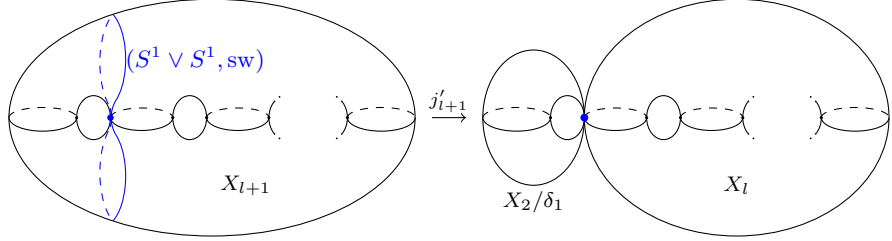
\begin{figure}[ht]\caption{Collapse a copy of $(S^1 \vee S^1,\sw)$ to obtain the wedge $X_2/\delta_1\vee X_l$.}\label{figurejtilde}
\centering
\begin{tikzpicture}[scale=0.45]
	
\draw[dashed,color=black!100] (0,0) arc (180:0:1 and 0.3);
	\draw[color=black!100](0,0) arc (180:360:1 and 0.4);
	
\draw[dashed,color=black!100] (3,0) arc (180:0:0.92 and 0.285);
	\draw[color=black!100] (3,0) arc (180:360:0.92 and 0.38);
	
\draw[dashed,color=black!100] (5.84,0) arc (180:0:0.92 and 0.285);
	\draw[color=black!100] (5.84,0) arc (180:360:0.92 and 0.38);
	
\draw[dashed,color=black!100] (10,0) arc (180:0:1 and 0.3);
	\draw[color=black!100] (10,0) arc (180:360:1 and 0.4);

\draw[black] (2.5,0) circle [x radius=0.5cm, y radius=0.65cm];;	
\draw[black] (5.34,0) circle [x radius=0.5cm, y radius=0.65cm];;
\draw[black] (7.68,0) arc (180:135:0.5 and 0.65);
\draw[dotted, semithick, color=black] (7.83,0.46) arc (135:100:0.5 and 0.7);
\draw[black] (7.68,0) arc (180:225:0.5 and 0.65);
\draw[dotted, semithick, black] (7.83,-0.46) arc (-135:-100:0.5 and 0.7);
\draw[black] (10,0) arc (0:45:0.5 and 0.65);
\draw[dotted, semithick,black] (9.85,-0.46) arc (-45:-80:0.5 and 0.7);
\draw[black] (10,0) arc (0:-45:0.5 and 0.65);
\draw[dotted, semithick, black] (9.85,0.46) arc (45:80:0.5 and 0.7);

        \draw[color=colour1] (3.02,0) arc (180:150: 1 and 1);
	\draw[dashed,color=colour1] (3.02,0.05) arc (0:20:1 and 1);
	
	\draw[color=colour1] (3.151,0.495) arc (-55:62:0.7 and 1.5);
	\draw[dashed,color=colour1] (2.937,0.40) arc (-130:-242:0.7 and 1.6);
	
	  \draw[color=colour1] (3.02,0) arc (180:150: 1 and -1);
	\draw[dashed,color=colour1] (3.02,-0.05) arc (0:20:1 and -1);
	
	\draw[color=colour1] (3.151,-0.495) arc (-55:62:0.7 and -1.5);
	\draw[dashed,color=colour1] (2.937,-0.40) arc (-130:-242:0.7 and -1.6);

\draw[black] (6,0) circle [x radius=6cm, y radius=3.5cm];

\filldraw[colour1] (3,0) circle (0.08);
\filldraw[black] (2,0) circle (0.03);
\filldraw[black] (4.84,0) circle (0.03);
\filldraw[black] (5.84,0) circle (0.03);
\filldraw[black] (7.68,0) circle (0.03);
\filldraw[black] (10,0) circle (0.03);

\draw[->] (12.5,0)--(13.5,0);

\draw[dashed,color=black!100] (0+14,0) arc (180:0:1 and 0.3);
	\draw[color=black!100](0+14,0) arc (180:360:1 and 0.4);
	
\draw[dashed,color=black!100] (3+14,0) arc (180:0:0.92 and 0.285);
	\draw[color=black!100] (3+14,0) arc (180:360:0.92 and 0.38);
	
\draw[dashed,color=black!100] (5.84+14,0) arc (180:0:0.92 and 0.285);
	\draw[color=black!100] (5.84+14,0) arc (180:360:0.92 and 0.38);
	
\draw[dashed,color=black!100] (10+14,0) arc (180:0:1 and 0.3);
	\draw[color=black!100] (10+14,0) arc (180:360:1 and 0.4);
	
\draw[black] (2.5+14,0) circle [x radius=0.5cm, y radius=0.65cm];;	
\draw[black] (5.34+14,0) circle [x radius=0.5cm, y radius=0.65cm];;
\draw[black] (7.68+14,0) arc (180:135:0.5 and 0.65);
\draw[dotted, semithick, color=black] (7.83+14,0.46) arc (135:100:0.5 and 0.7);
\draw[black] (7.68+14,0) arc (180:225:0.5 and 0.65);
\draw[dotted, semithick, black] (7.83+14,-0.46) arc (-135:-100:0.5 and 0.7);
\draw[black] (10+14,0) arc (0:45:0.5 and 0.65);
\draw[dotted, semithick,black] (9.85+14,-0.46) arc (-45:-80:0.5 and 0.7);
\draw[black] (10+14,0) arc (0:-45:0.5 and 0.65);
\draw[dotted, semithick, black] (9.85+14,0.46) arc (45:80:0.5 and 0.7);

\draw[black] (7.5+14,0) circle [x radius=4.5cm, y radius=3.5cm];
\draw[black] (1.5+14,0) circle [x radius=1.5cm, y radius=2cm];

\filldraw[colour1] (3+14,0) circle (0.1);
\filldraw[black] (2+14,0) circle (0.03);
\filldraw[black] (4.84+14,0) circle (0.03);
\filldraw[black] (5.84+14,0) circle (0.03);
\filldraw[black] (7.68+14,0) circle (0.03);
\filldraw[black] (10+14,0) circle (0.03);

\draw (7,-2) node {\footnotesize$X_{l+1}$};
\draw (21.5,-2) node {\footnotesize$X_{l}$};
\draw (15.5,-2.5) node {\footnotesize$X_{2}/\delta_1$};
\draw (13,0.5) node {\scriptsize$j_{l+1}'$};
\draw (5.5,1.7) node {\textcolor{colour1}{\small $(S^1 \vee S^1,\sw)$}};
\end{tikzpicture}
\end{figure}
Figure \ref{figurethemapj2} also shows that $j_2$ factors through the space $X_2/\delta_1$.  We therefore define $j_{l+1}$ to be the composition
\[X_{l+1}\xrightarrow{j_{l+1}'} X_2/\delta_1\vee X_l \xrightarrow{\tilde j\vee j_l} (X_{l+1})_\gamma \]
where $\tilde{j}$ is defined in Figure \ref{figurethemapj2}.
 \end{proof}
As an easy consequence of Proposition \ref{propgtrivialfortype1} we obtain the following homotopy equivalences
\begin{equation*}
\Sigma X \simeq  \Sigma\tilde X \vee \Sigma X_\gamma \mbox{ and } \Sigma \overline X \simeq  \Sigma\tilde X \vee \Sigma \overline X_\gamma.
\end{equation*}
In the following, we shall see that the factors $\Sigma X_\gamma$ and $\Sigma \overline X_\gamma$ give the factors $\overset{r-1}{\prod} \Omega U(n)$ and $\overset{r-1}{\prod} \Omega(U(n)/O(n))$ respectively in Theorem \ref{thmsplittingtype1fixedcircles} and that the factor $\Sigma \tilde X$ produces the factors $\Omega^2(U(n)/O(n))\times \prod_{r-1} \Omega O(n)$.
However, the map $j_r$ automatically induces a map 
\[\Mappe (X_\gamma,BU(n))\rightarrow \Mappe(X,BU(n);(0,0,\dots,0))\]
hence we only obtain a splitting on the level of mapping spaces in this trivial case. 

We now restrict to this trivial case for the rest of this section.  For the other cases, Proposition \ref{proprealpointedpathequal} will then give results for Theorem \ref{thmsplittingtype1fixedcircles} (1) and Propositions \ref{propsinglepointedcomponent} and \ref{propunpointedcomponentequiv} will give results for Theorem \ref{thmsplittingtype1fixedcircles} (2). We provide further decompositions at the level of the Real surface to continue the proof of Theorem \ref{thmsplittingtype1fixedcircles}. 

\begin{prop}\label{propreducetohesphere} Let $X_\delta$ be the $1$-cells in $\tilde X$ denoted by $\delta_2 , \dots,\delta_r$ then in the $\mathbb{Z}_2$-cofibration
\[ X_\delta \xrightarrow{\iota'} \tilde X \rightarrow (S^2,\he) \xrightarrow{\mu''} \Sigma(X_\delta)\]
the map $\mu''$ is $\mathbb{Z}_2$-nullhomotopic.
\end{prop}
\begin{proof}
The space $\tilde X$ is the quotient of a type $(r-1,r,0)$ Real surface with the one-cells denoted by $\gamma_2,\dots,\gamma_r$ collapsed to a point. Recall that the attaching map of $X$ is
\begin{equation}\label{eq:attachingmaptype1} \delta_1 \gamma_2 \delta_2 \gamma_2^{-1}\gamma_3 \delta_3 \gamma_3^{-1}\cdots \gamma_r \delta_r \gamma_r^{-1}   \end{equation}
and the induced attached map in $\tilde X$ becomes
\[\delta_1\delta_2\cdots \delta_r.\]  
We conclude that $\tilde X$ is a sphere $(S^2,\he)$ with $r$ of its fixed points identified.

Let $U$ denote the upper `hemisphere' of $\tilde X$; it is homeomorphic to a disc with $r$ of its boundary points identified and notice that $\tilde X=U \cup \bin (U)$. Now, there is a deformation retract $H\colon U\times I \rightarrow U$ of $U$ onto the wedge $\bigvee_{i=2}^{r}\delta_i$. Therefore, we define a left inverse to the map $\iota'$ via
\[x\mapsto 
\begin{cases}
H(x,1) & \mbox{for } x\in U \\
H(\sigma_{\tilde X}(x),1) & \mbox{for } x\in \sigma_{\tilde X}(U).
\end{cases}\]
and the result follows.
\end{proof}
We deduce that \[\Sigma \tilde X \simeq \Sigma X_\delta \vee \Sigma(S^2,\he).\]  The factor $ \Sigma X_\delta = \overset{r-1}{\bigvee} (S^1,\id)$ provides the factor $\overset{r-1}{\prod} \Omega O(n)$ for both cases in Theorem \ref{thmsplittingtype1fixedcircles}. We now show that the spaces $\Sigma(S^2,\he)$ and $\Sigma X_\gamma$ provide the other factors.
\begin{lemma}\label{claim:righthandfactor} There are homotopy equivalences
\begin{enumerate}\item $\Mappe(\Sigma X_\gamma,BU(n)) \simeq \prod_{r-1} \Omega (U(n)/O(n)); $
\item $\Mappe (\Sigma \overline X_\gamma,BU(n)) \simeq \prod_{r-1} \Omega U(n).$
\end{enumerate}
\end{lemma}
\begin{proof}
The space $\Sigma (X_\gamma)$ is the same as the wedge $ \bigvee_{r-1}\Sigma (S^1, \he)$. Looking at the $r$-pointed case, the $0$-skeleton of $\Sigma(X_\gamma)$ is collapsed and the space  $\Sigma\overline{(X_\gamma)}$ becomes the wedge $\Sigma\bigvee_{r-1}(S^1\vee S^1,\sw)$.  This shows part $(2)$ of the lemma.

For part $(1)$, we introduce a pullback similar to the pullbacks used in \cite{bairdmodulispace}. The space $\Mappe((S^1,\he),BU(n))$ fits into the following pullback diagram
\[ \xymatrix{\Mappe((S^1,\he),BU(n)) \ar[d]^{\tilde r} \ar[r]^-{\tilde u} & \Map^*( D^1, BU(n)) \ar[d]^r\\
\Mappe((S^0,\id),BU(n)) \ar[r]^-u & \Map^*(S^0,BU(n)).}\]
Here $\tilde r$ restricts to the fixed points of $(S^1,\he)$ and $\tilde u$ restricts to the upper hemisphere of $(S^1,\he)$  and then forgets about equivariance. Since \[\Mappe((S^0,\id),BU(n))\simeq BO(n) \] the map $u$ is just the inclusion $BO(n)\hookrightarrow BU(n)$ and hence the homotopy fibre of $u$ is $U(n)/O(n)$.  Since $r$ is a fibration, the square is also a homotopy pullback. We note that the space $\Map^*( D^1, BU(n))$ is contractible and so the result follows.
\end{proof}
\begin{lemma}\label{claimsphere}
There is a homotopy equivalence
\[\Mappe((S^2, \he),BU(n);(0,0))\simeq \Omega(U(n)/O(n))_0\]
where $\Omega (U(n)/O(n))_0$ denotes the connected component of $\Omega(U(n)/O(n))$ containing the basepoint.

\end{lemma}
 \begin{proof}
  There is a similar pullback as in Lemma \ref{claim:righthandfactor} 
 \[ \xymatrix{\Mappe((S^2,\he),BU(n)) \ar[d]^{\tilde r} \ar[r]^-{\tilde u} & \Map^*( D^2, BU(n)) \ar[d]^r\\
\Mappe((S^1,\id),BU(n)) \ar[r]^-u & \Map^*(S^1,BU(n)).}\]
 This time the map $u$ is homotopic to the inclusion $O(n) \hookrightarrow U(n)$ and so the homotopy fibre of $u$ is $\Omega (U(n)/O(n))$. The space $\Map^*(D^2, BU(n))$ is contractible and so there is an equivalence
 \[\Mappe((S^2, \he),BU(n))\simeq \Omega(U(n)/O(n))\]
 and the result follows.
 \end{proof}
 
 \begin{proof}[Proof of Theorem  \ref{thmsplittingtype1fixedcircles}] For $(1)$, it is enough to deal with the trivial component of $\Mappe(X,BU(n))$ by Proposition \ref{proprealpointedpathequal}.  Using a similar method  to the proof of Theorem \ref{thm:generalsplitting}, we have that Proposition \ref{propgtrivialfortype1} and Lemma  \ref{claim:righthandfactor} contribute the factor $\prod_{r-1} \Omega U(n)$, Proposition \ref{propreducetohesphere} contributes the factor $\prod_{r-1} \Omega O(n)$ and Lemma \ref{claimsphere} contributes the factor $\Omega^2(U(n)/O(n))$.
 
 For $(2)$, the proof is similar, but one has to be careful with the non-trivial components.
 \end{proof}
 \subsubsection{Case: \texorpdfstring{$r>0$, $a=1$}{r0,a1}} \label{subsectionr0a2}
 We use the techniques and notation of the previous section.  In particular, let $(P,\tilde\sigma)$ be a bundle of class $(0,0,\dots, 0)$ over a Real surface  $(X,\bin)$ of type $(g,r,1)$.  We first note that by Proposition \ref{propgtrivial} we can restrict to the cases
 \begin{equation}\label{eqncasesforg}
  g=r \hspace{10pt} \mbox{ or } \hspace{10pt} g=r+1.
 \end{equation}
 With these cases in mind, the main aim will be to prove the following theorem which is a restatement of Theorems \ref{thma}
 and \ref{thmb}
 for Real surfaces of type $(g,r,1)$.

 \begin{theorem}\label{thmtype2mainsplitting}
For the notation as above and $g$ as in (\ref{eqncasesforg}), there are homotopy equivalences
\begin{enumerate} \item $\g^*(P,\tilde\sigma)\simeq \g^*((g-r+1,1,1);(0,0))\times \prod\limits_{r-1} \Omega O(n) \times \prod\limits_{r-1} \Omega(U(n)/O(n));$
\item $\g^{*r+1}(P,\tilde\sigma)\simeq \g^{*2}((g-r+1,1,1);(0,0))\times \prod\limits_{r-1} \Omega O(n) \times \prod\limits_{r-1} \Omega U(n).$
\end{enumerate}
\end{theorem}
We note that after we have proven the above theorem, the only cases we have left to analyse will be gauge groups over Real surfaces of type $(2,1,1)$ and type $(1,1,1)$.

For the proof of the theorem, we will essentially follow the methods of the previous section.  Let $X_\gamma$ denote the sub-complex of $X$ consisting of the $1$-cells denoted by either $\gamma_i$ or $\bin( \gamma_i)$ for $2\leq i \leq r$.
  \begin{prop}\label{propgtrivialfortype2}
 Let $(X,\bin)$ be as above, then in the $\mathbb{Z}_2$-cofibration sequence
 \[X_\gamma \xrightarrow{\kappa} X \rightarrow \tilde X \xrightarrow{\nu} \Sigma(X_\gamma)\]
the map $\nu$ is $\mathbb{Z}_2$-nullhomotopic.
 \end{prop}
 
 \begin{proof}
 We define a left inverse to $\kappa$.  First in $X$ collapse the cells \[\gamma_{r+1}, \bin(\gamma_{r+1}), \delta_{r+1}, \bin(\delta_{r+1})\] and the cells $\gamma_{r+2}, \bin(\gamma_{r+2})$ if they exist.  We are left with a space $\mathbb{Z}_2$-homeomorphic to a Real surface of type $(r-1,r,0)$, we now use the map $j_r$ as defined in the proof of Proposition \ref{propgtrivialfortype1}.
 \end{proof}
The proof of the next proposition is identical to that of Proposition \ref{propreducetohesphere} except we exchange $(S^2,\he)$ for a Real surface $X'$ of type either $(2,1,1)$ or $(1,1,1)$.
 \begin{prop}\label{propgtrivialfixedtype2} Let $X_\delta$ be the $1$-cells in $\tilde X$ denoted by $\delta_2 , \dots,\delta_r$ then in the $\mathbb{Z}_2$-cofibration
\[ X_\delta \xrightarrow{\kappa'} \tilde X \rightarrow X' \xrightarrow{\nu'} \Sigma(X_\delta)\]
the map $\nu'$ is $\mathbb{Z}_2$-nullhomotopic. \qed
\end{prop}

\begin{proof}[Proof of Theorem \ref{thmtype2mainsplitting}]
This follows from Lemma \ref{claim:righthandfactor} and Propositions \ref{propgtrivialfixedtype2} and \ref{propgtrivialfortype2}.
\end{proof}

From Theorem \ref{thmtype2mainsplitting}, we reduce our study to the gauge groups 
\begin{gather*}\g^*((1,1,1);(0,0)) \mbox{ and } \g^*((2,1,1);(0,0)); \\ \g^{*2}((1,1,1);(0,0))\mbox{ and } \g^{*2}((2,1,1);(0,0)). \end{gather*}
The following theorem provides the remaining integral homotopy decompositions that we can obtain for these gauge groups. The theorem contributes to results in the last two rows of Theorem \ref{thma} and the last row in Theorem \ref{thmb}.

\begin{theorem}\label{thm211111integral}
There are integral homotopy equivalences
\begin{enumerate}
\item $\g^{*2}((1,1,1);(0,0))\simeq \g^{*}((1,1,1);(0,0))\times U(n)$;
\item $\g^{*2}((2,1,1);(0,0))\simeq \g^{*}((1,1,1);(0,0)) \times U(n) \times U(n) $;
\item $\g^{*}((2,1,1);(0,0)) \simeq \g^{*}((1,1,1);(0,0))\times U(n)$.
\end{enumerate}
\end{theorem}

We analyse the structure of a type $(2,1,1)$ Real surface $X'$.
\begin{prop}\label{prop211integral} Let $X'$ be a type $(2,1,1)$ Real Surface and let $X'_\gamma$ be the $1$-cells $\gamma_{2}, \gamma_{3},\bin(\gamma_{2}),\bin(\gamma_{3})$ of $X'$. Then in the $\mathbb{Z}_2$-cofibration
\[ X'_\gamma \xrightarrow{\kappa''} X' \rightarrow X'/X'_\gamma \xrightarrow{\nu''} \Sigma(X'_\gamma)\]
the map $\nu''$ is $\mathbb{Z}_2$-nullhomotopic.
\end{prop}

\begin{proof}
 We define a left inverse to $\kappa''$.  In $X'$ collapse the cell $\delta_{1}$ and then collapse a copy of $(S^1\vee S^1,\sw)$ so that $X'/\sim$ is the wedge $((\Sigma_1/\sim) \vee (\Sigma_1/\sim), \sw)$ where $(\Sigma_1/\sim)$ is a torus with $\delta_1$ collapsed. We now project to $(S^1\vee S^1, \sw)$ as we did in the proof of Proposition \ref{propgtrivialfortype1}, in fact, the left inverse is similar to the map $j_3$ from this proposition.
 \end{proof}
  In the following we show that the space $X'/X'_\gamma$ is $\mathbb{Z}_2$-homotopy equivalent to a $(1,1,1)$ Real surface $(X,\bin)$. We first recall the $\mathbb{Z}_2$-decomposition of $(X,\bin)$. The $0$-skeleton $X^0$ is given $3$ zero-cells $*_i$ for $1\leq i \leq 3$.  The one cells are then 
\begin{align*}
\delta_1,\delta,\sigma(\delta), \gamma_2,\sigma(\gamma_2)
\end{align*}
where the fixed circle is represented by $\delta_1$ and $\delta$ joins $*_{2}$ to $*_{3}$ and therefore $\delta\sigma(\delta)$ represents the copy of $(S^1,-\id)$.  The $1$-cell $\gamma_2$ joins $*_1$ to $*_2$ and $\bin(\gamma_2)$ joins $*_1$ to $*_3$.  One of the $2$ two-cells has attaching map
\[\delta_1\gamma_{2}\delta\sigma(\delta)\gamma^{-1}_{2}\]
and we define the other one equivariantly.

On the other hand, the space $X'/X'_\gamma$ has an induced $\mathbb{Z}_2$-complex structure as follows.  There is $1$ zero-cell $*$, to which we attach the one-cells
 \[\delta_1', \delta'\mbox{ and } \bin(\delta').\]
 There are $2$ two-cells, one of which is attached to the above $1$-skeleton via
 \[\delta_1\delta'\bin(\delta')\]
 and the other is glued equivariantly.
However, the sub-complex given by $\gamma_2\cup \bin(\gamma_2)$ of $(X,\sigma)$ is $\mathbb{Z}_2$-contractible and therefore $(X,\sigma)$ is homotopy equivalent to $\mathbb{Z}_2$-complex structure of $X'/X'_\gamma$.

 \begin{proof}[Proof of Theorem \ref{thm211111integral} (2) and (3)] By Proposition \ref{prop211integral}, we obtain the following homotopy equivalences
 \begin{align*}\Sigma X'&\simeq\Sigma X'_\gamma \vee\Sigma X'/X'_\gamma \\
 \Sigma \overline X'&\simeq\Sigma \overline X'_\gamma \vee\Sigma X'/X'_\gamma.
 \end{align*}
 In the first case the factor $\Sigma X'_\gamma$ is the same as the suspension of $(S^1\vee S^1,\sw)$.  We see that collapsing the $0$-skeleton of $\Sigma X'_\gamma$ provides the suspension of $\bigvee_2(S^1\vee S^1,\sw)$ and hence this corresponds to the factor $\Sigma \overline X'_\gamma$ in the second equivalence.  The result follows.
 \end{proof}
 
 \begin{proof}[Proof of Theorem \ref{thm211111integral} (1)]
 We use the $\mathbb{Z}_2$-structure provided after Proposition \ref{prop211integral}. In this $2$-pointed case, we identify the three $0$-cells $*_1,*_2,*_3$ to produce $\overline X$.  Let \[X_\gamma=\gamma_2 \cup \bin (\gamma_2)\] and let $\overline X_\gamma$ be the image in the quotient $\overline X$.  There is a left inverse to the inclusion 
 \[ \overline X_\gamma \hookrightarrow \overline X\]
 using a similar map to $j_2$ in the proof of Proposition \ref{propgtrivialfortype1}.  Therefore there is a homotopy equivalence
 \[\Sigma \overline X \simeq \Sigma \overline X_\gamma \vee \Sigma(\overline X/\overline X_\gamma)\]
 but by the comments after Proposition \ref{prop211integral} the factor $\Sigma(\overline X/\overline X_\gamma )$ is $\mathbb{Z}_2$-homotopy equivalent to the suspension of a Real surface of type $(1,1,1)$. This finishes the proof.
 \end{proof}
\subsubsection{Non-integral Decompositions}\label{subsectionnonintegraldecomp}
By the previous sections, we have reduced our study of the pointed gauge groups to those over Real Surfaces of the following types
\[(0,0,1) \mbox{, }(1,0,1) \mbox{ and } (1,1,1).\]
These spaces seem fundamental in some way and for the single-pointed case we do not obtain any further integral decompositions.  

However, one may expect these spaces to become easier to examine when we choose to invert $2$ since the involution has order $2$ and the $2$-torsion in $O(n)$ vanishes. This turns out to be the case and we will find that localising at a prime $p\neq2$ will prove particularly fruitful. 

In the coming sections we aim to prove Theorem \ref{thmc}, dealing with each part in turn.  The proof of each part is quite laborious, but we only provide full details for part $(1)$.  We outline the main parts of the proof of Theorem \ref{thmc} $(1)$

\begin{itemize}
 \item The existence of the pullback (\ref{diagramhomotopypullbackgg}) gives the existence of the map (\ref{eq:symmetricmatrixfactoring});
 \item We use an argument of \cite{Harrisonthehomotopygroupsoftheclassicalgroups} to prove that (\ref{eq:symmetricmatrixfactoring}) is a $p$-local homotopy equivalence for primes $p\neq 2$;
 \item We calculate the homotopy fibre of $qr$ in (\ref{eq:symmetricmatrixfactoring}).
\end{itemize}
The proofs of Theorem \ref{thmc} (2) and (3) will then invoke similar methods.

\vspace{0.15cm}

\noindent\textbf{Case: $(0,0,1)$}

\vspace{0.1cm}
Let $(S^2,-\id)$ be a Real surface of type $(0,0,1)$. By Proposition \ref{propsinglepointedcomponent}, all of the pointed gauge groups over $(S^2,-\id)$ are homotopy equivalent, so we assume that $(P,\tin)$ is of class $0$. In this section, we aim to prove the following theorem which is a restatement of Theorem \ref{thmc} $(1)$.

\begin{theorem}\label{thmsphereantipointsplitting}
Let $p\neq 2$ be prime and let $n$ be odd, then there is a $p$-local homotopy equivalence
\begin{equation*} \g^*(P,\tilde\sigma)\simeq_p \Omega (U(n)/O(n))\times \Omega^2 (U(n)/O(n)).\end{equation*}
\end{theorem}

Let $u \colon B\g^*(P,\tilde\sigma) \rightarrow \Map^{*2}(D^2,BU(n))$ be the map that restricts to the upper hemisphere of $(S^2,-\id)$ and forgets about equivariance.  Let \[r\colon B\g^*(P,\tilde\sigma) \rightarrow \Mappe((S^1\vee S^1,\sw),BU(n))\] be the map restricting to the $1$-skeleton of $(S^2,-\id)$.  These maps fit into the following pullback
\begin{equation}\label{diagramhomotopypullbackgg}\begin{gathered}\xymatrix{B\g^*(P,\tilde\sigma)\ar[r]^-{u} \ar[d]^-r & \Map^{*2}(D^2,BU(n)) \ar[d]^-{r'} \\ \Mappe((S^1\vee S^1,\sw),BU(n)) \ar[r]^-{u'} & \Map^*(S^1\vee S^1, BU(n))}\end{gathered}\end{equation}
where $r'$ restricts to the $1$-skeleton and $u'$ forgets about equivariance.

Let $\hat\varsigma\colon U(n)\rightarrow U(n)$ denote complex conjugation and note that $u'$ is homotopic to the map $\overline{\Delta}\colon U(n) \rightarrow U(n)\times U(n)$ where $\overline{\Delta}(\alpha)= (\alpha, \hat\varsigma{\alpha})$. Also note that the map $r'$ is homotopic to the map $\Delta^{-1}\colon U(n)\rightarrow U(n)\times U(n)$ where $\Delta^{-1}(\alpha)=(\alpha,\alpha^{-1})$. Let $Q$ be the strict pullback of $\overline{\Delta}$ and $\Delta^{-1}$ as in the following diagram
\begin{equation*}\begin{gathered}\xymatrix{ Q\ar[d]_{\pi_1} \ar[r]^{\pi_2} & U(n) \ar[d]^-{\Delta^{-1}} \\  U(n) \ar[r]^-{\overline{\Delta}} & U(n)\times U(n).}\end{gathered}\end{equation*}
We will see that $Q$ retracts off $B\g^*(P,\tin)$ after inverting the prime $2$.

The map $r'$ in diagram (\ref{diagramhomotopypullbackgg}) is a fibration, and hence this diagram is a homotopy pullback.  Therefore, there is an induced homotopy commuting diagram
\begin{equation}\label{diagram:pullbackmap} \begin{gathered}\xymatrix{Q \ar@/_/[ddr]_{\pi_1}\ar@/^/[drr]^{\pi_2} \ar@{.>}[dr]^{\tilde\pi}\\&B\g^*(P,\tilde\sigma)\ar[r]^-{u} \ar[d]^-r & U(n) \ar[d]^-{\Delta^{-1}} \\ & U(n) \ar[r]^-{\overline{\Delta}} & U(n)\times U(n)}\end{gathered}\end{equation}
where we have replaced the pullback square (\ref{diagramhomotopypullbackgg}) with a homotopy equivalent square.
\begin{lemma}\label{lemmaQhomeomorphictoU(n)/O(n)}
 The pullback $Q$ is homeomorphic to $U(n)/O(n)$.
\end{lemma}
\begin{proof}
 The pullback $Q$ is the space \[\{ A\in U(n)\mid A^{-1}=\hat\varsigma(A)\}.\] 
 Let $f\colon U(n)\rightarrow Q$ be defined by $f(A)=A\hat\varsigma(A)^{-1}$. For matrices $A\in U(n)$ and $W\in U(n)^{\hat\varsigma}=O(n)$ we have
 \[ (AW)\hat\varsigma(AW)^{-1}=AW\hat\varsigma(W^{-1})\hat\varsigma(A^{-1})=A\hat\varsigma(A)^{-1}\]
 since $\hat\varsigma$ is a homomorphism. Hence $f$ induces a map $f'\colon U(n)/O(n)\rightarrow Q$.  
 
 We show that $f'$ is a bijection. For injectivity,  let $A,B \in U(n)$ and suppose that $A\hat\varsigma (A)^{-1}=B\hat\varsigma(B)^{-1}$, then
 \[ I_n=B^{-1} A\hat\varsigma (A)^{-1} \hat\varsigma(B)=(B^{-1} A)\hat\varsigma({B}^{-1}A)^{-1}\]
 for $I_n\in U(n)$ the identity matrix.  Hence $B^{-1}A\in U(n)^{\hat\varsigma}$ and so $AU(n)^{\hat\varsigma} \equiv BU(n)^{\hat\varsigma}$.  
 
For surjectivity, let $A\in Q$ then $A$ is symmetric and, due to the Autonne–Takagi factorisation (see \cite{youlaanormalform}), there is a unitary matrix $P$ such that $A=PDP^t$ where $D$ is a diagonal matrix with real entries.  Let $\sqrt{D}$ be a diagonal (hence an element of $Q$) matrix in $U(n)$ such that $\sqrt{D}^2 =D$.  We have \[A=P\sqrt{D} \sqrt{D} P^t=P\sqrt{D}\hat\varsigma(P\sqrt{D})^{-1}\] and therefore $f'((P\sqrt{D})O(n))=A$.

The map $f'$ is therefore a continuous bijection, and since $U(n)/O(n)$ is compact and $Q$ is Hausdorff it is a homeomorphism.
\end{proof}

The above diagram and Lemma \ref{lemmaQhomeomorphictoU(n)/O(n)} give the following composition
\begin{equation}\label{eq:symmetricmatrixfactoring} \varphi \colon U(n)/O(n) \xrightarrow{f'} Q \xrightarrow{\tilde \pi} B\g^*(P,\tilde\sigma) \xrightarrow{r} U(n) \xrightarrow{q} U(n)/O(n)\end{equation}
for $q$ the quotient map.  From the properties of $\pi_1$ we see that $\varphi$ is homotopic to a map that sends an element $AO(n)$ to $AA^tO(n)$. For odd $n$, \cite{Harrisonthehomotopygroupsoftheclassicalgroups} showed that the related map 
\begin{equation}\label{eqnharrismap}\begin{gathered}\begin{aligned} SU(n)/SO(n)& \rightarrow SU(n)/SO(n) \\ ASO(n)& \mapsto AA^tSO(n)\end{aligned}\end{gathered}\end{equation}
is a homotopy equivalence when localised at a prime $p\neq 2$. Our aim is to show that the same is true for $\varphi$.

\begin{lemma}\label{lemmaunonunson}
For a prime $p\neq 2$, there is an $p$-local homotopy equivalence
\[U(n)/O(n)\simeq_p U(n)/SO(n).\]
\end{lemma}

\begin{proof}
Consider the following pullback diagram where the downward arrows represent taking universal covers
\[ \xymatrix{ U(n)/SO(n) \ar[r] \ar[d] & BSO(n)\ar[d] \ar[r] & BU(n) \ar@{=}[d] \\ U(n)/O(n) \ar[r] \ar[d] & BO(n) \ar[d] \ar[r] & BU(n) \\ K(\mathbb{Z}_2, 1) \ar@{=}[r] & K(\mathbb{Z}_2, 1)}.\]
The result immediately follows.
\end{proof}
We now show that $U(n)/SO(n)$ further decomposes into the product \[SU(n)/SO(n)\times S^1.\]The map $BSO(n)\rightarrow BU(n)$ factors through $BSU(n)$. Hence, we obtain the following commutative diagram which defines the maps $i$ and $j$
\begin{equation}\label{eqnunsondecomposition} \begin{gathered}\xymatrix{ & U(n)\ar[d] \ar@{=}[r] & U(n) \ar[d]^-{f} \\ SU(n)/SO(n)\ar@{=}[d] \ar[r]^-i & U(n)/SO(n) \ar[d]\ar[r]^-j & S^1 \ar[d] \\ SU(n)/SO(n) \ar[r] & BSO(n) \ar[d] \ar[r] & BSU(n) \ar[d] \\ & BU(n) \ar@{=}[r] & BU(n).}\end{gathered}\end{equation}
It is not too much more work to show the following lemma.

\begin{lemma}\label{lemmaunsonsunsons1}
There is a homotopy equivalence
\[ \eta\colon  SU(n)/SO(n)\times S^1\xrightarrow{\simeq} U(n)/SO(n) .\]
\end{lemma}
\begin{proof}
There is a right inverse $l$ to the map $f$ and there is an action of $U(n)$ on $U(n)/SO(n)$, hence the composition
\[\eta \colon S^1 \times SU(n)/SO(n) \xrightarrow{l\times i} U(n) \times U(n)/SO(n) \xrightarrow{\mbox{\footnotesize `action'}} U(n)/SO(n)\]
is the required homotopy equivalence.
\end{proof}

Let $\varphi$ be the composition in equation (\ref{eq:symmetricmatrixfactoring}) and then define \[s\colon U(n)/SO(n) \rightarrow U(n)/SO(n)\] to be the composition
\[U(n)/SO(n)\xrightarrow{\simeq} U(n)/O(n) \xrightarrow{\varphi} U(n)/O(n)\xrightarrow{\simeq}  U(n)/SO(n)\]
Our aim is to show that $s$ restricts to the factors $SU(n)/SO(n)$ and $S^1$ in a nice enough way.  
\begin{lemma}\label{lemmathereexistmaps}
There exist maps \begin{align*}s''\colon SU(n)/SO(n) &\rightarrow SU(n)/SO(n) \\ s'\colon S^1&\rightarrow S^1\end{align*} such that the following is a homotopy commuting square
\[\xymatrix{SU(n)/SO(n)\times S^1 \ar[r]^{s''\times s'} \ar[d]^{\eta}& SU(n)/SO(n)\times S^1 \ar[d]^\eta\\
U(n)/SO(n)  \ar[r]^s & U(n)/SO(n).} \]
Furthermore, these maps can be chosen such that $s''$ is homotopic to the map
\[ ASO(n)\mapsto AA^t SO(n)\]
and $s'$ is homotopic to the map $x\mapsto x^2$.
\end{lemma}
\begin{proof}
Let $\tilde s\colon SU(n)/SO(n)\times S^1 \rightarrow SU(n)/SO(n)\times S^1$ be the composition 
\[ SU(n)/SO(n)\times S^1 \xrightarrow{\eta} U(n)/SO(n) \xrightarrow{s} U(n)/SO(n) \xrightarrow{\eta^{-1}} SU(n)/SO(n)\times S^1\]
for a homotopy inverse $\eta^{-1}$ of $\eta$.  Let $\iota\colon SU(n)/SO(n)\rightarrow SU(n)/SO(n) \times S^1$ and $\kappa \colon S^1 \rightarrow SU(n)/SO(n)\times S^1$ be the inclusions.  We note that $\iota$ is homotopic to
\[SU(n)/SO(n)\xrightarrow{i} U(n)/SO(n)\xrightarrow{\eta^{-1}} SU(n)/SO(n)\times S^1\]
where $i$ is as in diagram (\ref{eqnunsondecomposition}).
By the way the homotopy equivalences are defined in Lemmas \ref{lemmaunonunson} and \ref{lemmaunsonsunsons1}, we see that the composition $si$ is homotopic to 
\[BSO(n)\mapsto BB^tSO(n) \mbox{ for } B\in SU(n)\]
and hence the image of this map lands in the image of $i$. We deduce that $\tilde s \iota$ has image in $SU(n)/SO(n)$ and we define
\[s''= \tilde s \iota.\]
 Similarly $\tilde s\kappa$ has image in $S^1$ and we define $s'=\tilde s\kappa$. We see that $s''$ is homotopic to a map defined via $ASO(n)\mapsto AA^tSO(n)$ and that $s'$ is homotopic to the map $x\mapsto x^2$.
\end{proof}
We immediately obtain the following homotopy commuting diagram where the rows are homotopy fibrations
\begin{equation}\label{eq:restrictstofactors} \begin{gathered}\xymatrix{SU(n)/SO(n) \ar[d]^{s''} \ar[r]^i &U(n)/SO(n) \ar[d]^{s} \ar[r]& S^1\ar[d]^{s'} \\ SU(n)/SO(n) \ar[r]^i & U(n)/SO(n) \ar[r] & S^1}\end{gathered}\end{equation}
 By Lemma \ref{lemmathereexistmaps}, the map $s''$ is homotopic to the map in equation (\ref{eqnharrismap}) and hence it is a $p$-local equivalence when $n$ is odd and $p\neq 2$ is a prime. We note that $s'$ is also a $p$-local equivalence. Finally, the spaces in $(\ref{eq:restrictstofactors})$ are connected, hence $s$ is also a $p$-local equivalence.  We are now able to deduce the following.

\begin{prop}\label{proppointedtype0retract}
With notation as in equation (\ref{eq:symmetricmatrixfactoring}), let $F$ be the homotopy fibre of $qr\colon B\g^*(P,\tilde\sigma)\rightarrow U(n)/O(n)$. Then for $n$ odd and for any prime $p\neq 2$, there is a $p$-local homotopy equivalence
\begin{equation*}{\g^*(P,\tilde\sigma)}\simeq_p \Omega(U(n)/O(n))\times \Omega F.\end{equation*}
\end{prop}
\begin{proof}
 Recall the maps $f'$ and $\tilde \pi$ from equation ($\ref{eq:symmetricmatrixfactoring}$).  Then the above discussion has shown that $\tilde \pi f'$ provides a $p$-local homotopy section to the homotopy fibration
 \[ F \rightarrow B\g^*(P,\tilde\sigma)\xrightarrow{qr} U(n)/O(n)\]
 and the result follows.
\end{proof}

Therefore, to prove Theorem \ref{thmsphereantipointsplitting} it only remains to identify the fibre $F$. 
\begin{prop} \label{thmidentifyF}
For any prime $p\neq 2$, there is a $p$-local homotopy equivalence
\[ F\simeq_p \Omega (U(n)/O(n)).\]
\end{prop}

\begin{proof}
The map $qr$ from equation (\ref{eq:symmetricmatrixfactoring}) is defined as a composition, hence there is a homotopy commutative diagram
\[\xymatrix{F \ar[d] \ar[r] & B\g^*(P,\tilde\sigma)  \ar[r]^{qr} \ar[d]^r & U(n)/O(n) \ar@{=}[d] \\ O(n) \ar[r] & U(n) \ar[r]^-q & U(n)/O(n)}\]
where the left square is a homotopy pullback square. The map $r$ is a fibration since it is induced by $i\colon (S^1,-\id) \hookrightarrow (S^2,-\id)$; the inclusion of the meridian copy of $(S^1,-\id)$ into $(S^2,-\id)$. Therefore the space $F$ is homotopy equivalent to the strict pullback of $O(n)\rightarrow U(n) \xleftarrow{r} B\g^*(P,\tin)$, which is the relative mapping space \[\Mappe\Big(\big((S^2,-\id),(S^1,-\id)\big),\big(BU(n),BO(n)\big);0\Big).\]

We will associate another pullback square with this description of $F$.  There is a map $T\colon F\rightarrow \Mappe((S^2,-\id),(BU(n),\id);0)$ given by
\[T(f)(x)=\begin{cases} f(x) & \text{for } x \text{ in the upper hemisphere including equator;} \\
f(-\id(x)) & \text{for } x \text{ in the lower hemisphere excluding equator.} \\
\end{cases} \]
Let $i\colon (S^1,-\id) \hookrightarrow (S^2,-\id)$ be defined as above, then $i$ induces the following homotopy pullback diagram
\[ \xymatrix{ F\ar[r]^-T \ar[d]_{i^*} \pullbackcorner &\Mappe((S^2,-\id),(BU(n),\id);0) \ar[d]^{i^*} \\ O(n) \ar@{^{(}->}[r] & U(n)}.\]
There is a homeomorphism \[\Mappe((S^2,-\id),(BU(n),\id);0)\cong \Map^*(\mathbb{R}P^2,BU(n);0)\]
but for a prime $p\neq 2$ the space $\mathbb{R}P^2$ is $p$-locally contractible.  Therefore, $p$-locally, we have identified the space $F$ as the fibre of the inclusion $O(n)\rightarrow U(n)$ and the result follows.
\end{proof}
\begin{proof}[Proof of Theorem \ref{thmsphereantipointsplitting}]
Use Propositions \ref{proppointedtype0retract} and \ref{thmidentifyF}.
\end{proof}

\vspace{0.15cm}

\noindent\textbf{Case: $(1,0,1)$}

\vspace{0.1cm}
Let $(T,\tau)$ be a Real surface of type $(1,0,1)$ and since all pointed gauge groups over $(T,\tau)$ are homotopy equivalent, we restrict to the case where $(P,\tilde\sigma)$ is a bundle of class $0$ over $(T,\tau)$.  We will use similar techniques to the even genus case to obtain the following theorem, which is a restatement of Theorem \ref{thmc} $(2)$.
\begin{theorem}\label{thmtorusantipointsplitting}
For any prime $p\neq 2$ and $n$ odd, there is a $p$-local homotopy equivalence
\begin{equation*} \g^*(P,\tilde\sigma)\simeq_p \Omega (U(n)/O(n))  \times \Omega^2 (U(n)/O(n)) \times \Omega U(n).\end{equation*}
\end{theorem}
\begin{proof}
Let $u \colon B\g^*(P,\tilde\sigma) \rightarrow \Map^*(C,BU(n))$ be the map that forgets about equivariance and restricts to the upper half of $(T,\tau)$ which is homeomorphic to a cylinder $C$.  Let $i$ be the inclusion of the boundary circles of $C$, then $i$ induces a pullback
\begin{equation}\label{eqnhomotopypullback101}\begin{gathered}\xymatrix{B\g^*(P,\tilde\sigma)\ar[r]^-{u} \ar[d]^-r & \Map^{*}(C,BU(n)) \ar[d]^-{r'} \\ \Mappe((S^1\vee S^1,\sw),BU(n)) \ar[r]^-{u'} & \Map^*(S^1\amalg S^1, BU(n))}\end{gathered}\end{equation}
where $r'=i^*$ and $r$ is the restriction to the one-skeleton of $(X,\bin)$.

In a similar fashion to the way we obtained diagram $(\ref{diagram:pullbackmap})$, we replace $(\ref{eqnhomotopypullback101})$ with a homotopy equivalent square and obtain the diagram
\begin{equation*}\begin{gathered}\xymatrix{Q \ar@/_/[ddr]\ar@/^/[drr] \ar@{.>}[dr]\\&B\g^*(P,\tilde\sigma)\ar[r]^-{u} \ar[d]^-r & U(n) \ar[d]^-{\Delta^{-1}} \\ & U(n) \ar[r]^-{\overline{\Delta}} & U(n)\times LBU(n).}\end{gathered}\end{equation*}
Here $LBU(n)$ denotes the free loop space of $U(n)$ and $Q$ is the strict pullback of the diagram
\[ U(n) \xrightarrow{\overline\Delta} U(n)\times LBU(n) \xleftarrow{\Delta^{-1}} U(n)\]
and hence $Q$ is again the symmetric matrices in $U(n)$. We deduce that $U(n)/O(n)$ also $p$-locally retracts off $B\g^*(P,\tilde\sigma)$. 

It is clear that, as in the even case, there is similar description for the fibre $F$ of the map $B\g^*(P,\tilde\sigma) \rightarrow U(n)/O(n)$. The space $F$ fits into the following pullback diagram
\[ \xymatrix{ F\ar[r] \ar[d] \pullbackcorner &\Mappe((T,\tau),(BU(n),\id);0) \ar[d]^{\bar{r}} \\
O(n) \ar@{^{(}->}[r] & U(n).}\]

We note that if we let $K$ be a Klein bottle then there is an homeomorphism \[\Mappe((T,\tau),(BU(n),\id);0)\cong \Map^*(K,BU(n);0)\]
The map $\bar{r}$ is induced by the inclusion $S^1\hookrightarrow K$ which on fundamental groups induces the quotient 
\begin{align*}\mathbb{Z} & \rightarrow  \mathbb{Z}\times \mathbb{Z}_2 \\ a & \mapsto (0,[a]_2)\end{align*} onto the right factor.  We see that for a prime $p\neq 2$, the map $\bar{r}$ is $p$-locally nullhomotopic and we obtain
\[ \Omega F \simeq_p \Omega^2 (U(n)/O(n)) \times \Omega\Map^*(K,BU(n);0).\]
Now for $p\neq 2$ prime we have a $p$-local homotopy equivalence $K\simeq_p S^1$ because $K$ is a $K(\mathbb{Z}\times \mathbb{Z}_2,1)$. Therefore, the space $\Omega \Map^*(K,BU(n);0)$ is homotopy equivalent to $\Omega U(n)$ when localised away from $2$ and Theorem \ref{thmtorusantipointsplitting} follows.
\end{proof}
\vspace{0.15cm}

\noindent\textbf{Case: $(1,1,1)$}

\vspace{0.1cm}
Let $(X,\bin)$ be a Real surface of type $(1,1,1)$.  For convenience we choose $(P,\tilde\sigma)$ to be a bundle of class $(0,0)$ over $(X,\bin)$. We use a very similar method to the previous sections to prove the following theorem. This theorem is a more general statement than Theorem \ref{thmc} $(3)$, whose statement claims to only be valid for odd $n$.

\begin{theorem}\label{thm111pointed}
For any prime $p\neq 2$, there is a $p$-local homotopy equivalence
\begin{equation*} \g^*(P,\tilde\sigma)\simeq_p \g^*((S^2,-\id);0)\times \Omega O(n).\end{equation*}
\end{theorem}

\begin{proof}
We first recall the $\mathbb{Z}_2$-decomposition of $(X,\bin)$. The $0$-skeleton $X^0$ is given $3$ zero-cells $*_i$ for $1\leq i \leq 3$.  The one cells are then 
\begin{align*}
\delta_1,\delta,\sigma(\delta), \gamma_2,\sigma(\gamma_2)
\end{align*}
where the fixed circle is represented by $\delta_1$ and $\delta$ joins $*_{2}$ to $*_{3}$ and therefore $\delta\sigma(\delta)$ represents the copy of $(S^1,-\id)$.  The $1$-cell $\gamma_2$ joins $*_1$ to $*_2$ and $\bin(\gamma_2)$ joins $*_1$ to $*_3$.  One of the $2$ two-cells has attaching map
\[\delta_1\gamma_{2}\delta\sigma(\delta)\gamma^{-1}_{2}\]
and we define the other one equivariantly.

Since the subspace $\gamma_2\cup\bin(\gamma_2)$ is $\mathbb{Z}_2$-contractible, we amend the above decomposition to have only $3$ one-cells $\delta_1, \delta, \bin(\delta)$ and amend the attaching map to
\[\delta_1\delta\bin(\delta).\]

We obtain a pullback similar to that of the previous section
\[\xymatrix{B\g^*(P,\tilde\sigma)\ar[r]^-{u} \ar[d]^-r & \Map^{*3}(D^2,BU(n)) \ar[d]^-{r'} \\ \Mappe((S^1,\id)\vee (S^1\vee S^1,\sw),BU(n);w_1) \ar[r]^-{u'} & \Map^*(S^1\vee S^1\vee S^1, BU(n))}\]
where $r$ is the restriction to the one-skeleton of $(X,\bin)$ and $u$ restricts to one of the two-cells and forgets about equivariance.

In a similar fashion to the way we obtained diagram (\ref{diagram:pullbackmap}), we obtain the diagram
\begin{equation}\label{eqn111pullback}\xymatrix{O(n) \ar@/_/[ddr]_{f_1} \ar@/^/[drr]^{f_2} \ar@{.>}[dr]^{f_3}\\&B\g^*(P,\tilde\sigma)\ar[r]^-{u} \ar[d]^-r & U(n)\times U(n) \ar[d]^-{r'} \\ & SO(n)\times U(n) \ar[r]^-{u'} & U(n)\times U(n)\times U(n)}\end{equation}
where $f_1, f_2$ and $f_3$ are to be defined momentarily.

The map $r'\colon U(n) \times U(n) \rightarrow U(n) \times U(n) \times U(n)$ is the map
\[r'(A,B)=(B^{-1}A^{-1},A,B)\]
and the map $u'\colon SO(n)\times U(n) \rightarrow U(n)\times U(n) \times U(n)$ is the map
\[u'(C,D)=(C,D, \overline{D}).\]
We can therefore define maps $f_1 \colon O(n) \rightarrow SO(n)\times U(n)$ and $f_2\colon O(n) \rightarrow U(n) \times U(n)$ by
\[ f_1(X)=(X^{-2}, X) \mbox{ and } f_2(Y)=(Y,Y)\]
such that $u'f_1=r'f_2$.  Since \ref{eqn111pullback} is a homotopy pullback, there exists a map \[f_3 \colon O(n) \rightarrow B\g^*(P,\tilde\sigma)\] such that the composition
\[ \chi \colon O(n) \xrightarrow{f_3} B\g^*(P,\tilde\sigma) \xrightarrow{r} O(n)\times U(n)\xrightarrow{p_1}O(n)\]
sends an element $X$ to $X^{-2}$.  Then observe that $\chi$ has image lying in $SO(n)$ and therefore when $\chi$ is restricted to $SO(n)$, it is the inverse of the $H$-space squaring map.  We conclude that restriction of $\chi$ to $SO(n)$ is a $p$-local homotopy equivalence for $p\neq 2$ and therefore $SO(n)$ retracts off $B\g(P,\tilde\sigma)$.  

The map $p_1r$ is just the restriction to the fixed points of the involution.  Hence the fibre of this map is the space $B\g^*((0,0,1);0)$ which we have already studied. We finish by noting that $\Omega SO(n)$ and $\Omega O(n)$ are homeomorphic.
\end{proof}
\subsection{Unpointed Gauge Groups}\label{sectionunpointedgaugegroups}
In the last section, we showed that certain trivialities of the attaching map of the top cells of $X$ led to homotopy decompositions in the pointed case. We will see that these decompositions somewhat extend to the unpointed case.

\subsubsection{Integral Decompositions}\label{subsectionunpointedintegraldecomp}
Let $(X,\bin)$ be a Real surface of type $(g,r,a)$.  In the following proposition $g'$ will denote the number of $\alpha_i$ and $\beta_i$ cells in the description of $(X,\sigma)$ in Section \ref{sectionkleinsurfacesasz2complexes}. Explicitly
\[g'=\begin{cases} 
(g-r+1)  &\mbox{ when } a=0; \\
(g-r)  & \mbox{ when } a=1\mbox{ and } g-r \mbox{ even;} \\
(g-r-1)  &\mbox{ when } a=1\mbox{ and } g-r \mbox{ odd.} \\
\end{cases}\]
We now present Proposition \ref{propintegralunpointed} which is a restatement of Theorem \ref{thmd} $(1)$.  

\begin{prop}\label{propintegralunpointed}
There are homotopy equivalences
\[ \g((g,r,a);(c,w_1,\dots,w_r))\simeq \g((g-g',r,a);(c,w_1,\dots,w_r))\times \prod_{g'} \Omega U(n).\]
\end{prop}

\begin{proof}
In essence, we follow a proof in \cite{toddprimary}.  For convenience, we write \[(c,\overline w):=(c,w_1,\dots,w_r).\] Let $X_{\alpha\beta}=\bigvee_{g'} (S^1\vee S^1,\sw)$ be sub-complex of $X$ represented by $\alpha_i, \sigma(\alpha_i), \beta_i, \sigma(\beta_i)$. Recall the $\mathbb{Z}_2$-cofibration sequence of Proposition \ref{propgtrivial}
\[X_{\alpha\beta} \hookrightarrow X \xrightarrow{q} X' \xrightarrow{\mu} \Sigma(X_{\alpha\beta}).\]
Then the map $q$ induces the following diagram
\[\xymatrix{
\Omega B \ar@{=}[d] \ar[r]^-{\partial_{(c,\overline w)}} & \Mappe(X',BU(n);(c, \overline w))\ar[d]^{q^*} \ar[r]&  \Mape(X',BU(n); (c,\overline w)) \ar[r]^-{ev} \ar[d]^{q^*} & B\ar@{=}[d]\\
\Omega B \ar[r]^-{\varphi_{ (c, \overline w)}} & \Mappe(X,BU(n);(c, \overline w)) \ar[r] &   \Mape(X,BU(n);(c, \overline w))\ar[r]^-{ev} & B}\] 
where \[B=\begin{cases}BU(n) \mbox{ if } r=0; \\ BO(n) \mbox{ otherwise.}\end{cases} \]
The fact that $\varphi_{(c, \overline w)}=q^*\partial_{(c, \overline w)}$ obtains the following diagram which defines the maps $h$ and $h'$
\begin{equation*}\label{eqintegralsplitting}\xymatrix{ & \Map^*(\Sigma(X),BU(n);(c, \overline w))\ar@{=}[r] \ar[d] &  \Mappe(\Sigma X ,BU(n);(c, \overline w))\ar[d]^{(\Sigma i)^*}  \\
\g(g-g')\ar[r]^{h'}\ar@{=}[d] &  \g((g,r,a);(c, \overline w))\ar[r]^-h \ar[d] & \Mappe(\Sigma (X_{\alpha\beta}),BU(n))\ar[d]^{\mu^*}\\
 \g(g-g')\ar[r]&   \Omega B \ar[r]^-{\partial_{(c, \overline w)}} \ar[d]^-{\varphi_{(c, \overline w)}} & \Mappe(X',BU(n);(c, \overline w)) \ar[d]^{q^*} \\
  & \Mappe(X,BU(n);(c, \overline w)) \ar@{=}[r] & \Mappe(X,BU(n);(c, \overline w))
 }\end{equation*}
and in which $\g(g-g'):=\g((g-g',r,a);(c,\overline w))$. By Proposition \ref{propgtrivial} the map $\mu^*$ is trivial. Hence there is a section to the map $ (\Sigma i)^*$, so there is also a section to $h$ and the result follows.
\end{proof}

The quotient map $q$ in Proposition \ref{propintegralunpointed} induced an isomorphism on $\pi_0$ between \[\Mappe(X,BU(n);(c, \overline w)) \mbox{ and } \Mappe(X',BU(n);(c, \overline w)).\]
However, for a fixed cell $\delta_i$ of $(X,\sigma)$, the quotient map $\tilde q\colon X\rightarrow X/\delta_i$ automatically induces the map 
\[ \Mape(X/\delta_i,BU(n))\xrightarrow{q^*} \Mape(X,BU(n);0)\]
hence the requirement for $w_i=0$ in Theorem \ref{thmd} (3).
Whilst there is an equivalence
\[\Mappe(X,BU(n);(c,0))\simeq \Mappe(X,BU(n);(c,1))\]
there is not necessarily an equivalence in the unpointed case in general.  Hence, there is not enough information to guarantee the commutativity of the diagram needed to induce a homotopy decomposition.  

Omitting such non-trivialities allows further splittings; let $X_1$ be a subset of the $1$-cells of $X$ such that 
\begin{enumerate}
\item if there is a fixed cell $\delta_i \subset X_1$ then $w_i=0$ and;
\item for appropriate components the induced map
\[g^*\colon \Mappe(\Sigma X_1,BU(n);(\overline w))\rightarrow\Mappe (X/X_1,BU(n);(c,\overline w))\] 
is $\mathbb{Z}_2$-nullhomotopic.
\end{enumerate}
Under these assumptions, it is clear that the methods in the previous proposition would yield further homotopy decompositions.

\begin{proof}[Proof of Theorem \ref{thmd} \textit{(2) and (3)}.]\label{proofoftheoremd3}
The above conditions apply to the $1$-cells considered in Propositions \ref{propgtrivialfortype1}, \ref{propgtrivialfortype2} and \ref{prop211integral} for bundles of arbitrary type. 

Additionally the conditions are satisfied by the $1$-cells considered in Propositions \ref{propreducetohesphere} and \ref{propgtrivialfixedtype2} for bundles of type $(c,w_1,0,\dots,0)$. 
When $n$ is odd we can take advantage of Proposition \ref{propunpointedcomponentequiv} to obtain the table in Theorem \ref{thmd} $(3)$. We have now finished the proof of Theorem \ref{thmd}.
\end{proof}

\subsubsection{Analysing the Boundary Map}\label{subsectionanalysingboundary} 
Let $(P,\tilde\sigma)\rightarrow (X,\bin)$ be a Real bundle of class $(c,w_1,\dots,w_r)$ over a Real surface $(X, \sigma)$ of type $(g,r,a)$. Let 
\[B=\begin{cases}
BO(n) & \mbox{if } r>0 \\
BU(n) & \mbox{otherwise}
\end{cases}\]
and consider the homotopy fibration sequence induced from the map that evaluates at the basepoint of $X$
\begin{equation}\label{equevaluationfibration} \g(P,\tin) \rightarrow \Omega B \xrightarrow{\partial_P} \Map_{\mathbb{Z}_2}^*(X, BU(n);P)\rightarrow \Map_{\mathbb{Z}_2}(X, BU(n);P) \rightarrow B\end{equation}

Since $\g(P,\tin)$ appears as the homotopy fibre of the boundary map $\partial_P$, we aim to gather information about $\g(P,\tin)$ by studying $\partial_P$. Our method will involve comparing $\partial_P$ to a map arising from a similar homotopy fibration sequence found in \cite{tggoverriemannsurfaces}. This approach is particularly fruitful when $X^\sigma$ is nonempty, that is, when $r > 0$. We reserve analysis of the $r=0$ cases unhandled by Section \ref{subsectionunpointedintegraldecomp} to later sections, however, we will require discussion from this section and Section \ref{subsectionnonintegraldecomp}.

Note that \[\pi_0\big(\Map(S^2,BU(n))\big)\cong \mathbb{Z}\] and for $d\in \mathbb{Z}$ we obtain a fibration sequence
\begin{equation}\label{equnoneqeval}U(n) \xrightarrow{\partial_d} \Map^*(S^2,BU(n);d) \rightarrow \Map(S^2,BU(n);d) \rightarrow BU(n).\end{equation}
The trivialities of the map $\partial_d$ were extensively studied in \cite{tggoverriemannsurfaces}.  We state the relevant results from this paper.

\begin{theorem}[Theriault]\label{thm:tgghomotopydecomp}
Let $p$ be a prime and let 
\[\partial_d \colon U(n) \rightarrow \Map^*(S^2,BU(n);d)\]
be as in (\ref{equnoneqeval}). Then
  \begin{enumerate}
  \item if $p\nmid n$, then $\partial_d$ is $p$-locally trivial.
  \item if $n=p$ with $p\mid d$, then $\partial_d$ is $p$-locally trivial.
 \end{enumerate}
\end{theorem}

 The case $n = p \nmid d$ was also studied in \cite{tggoverriemannsurfaces}; the map $\partial_d$ is not $p$-locally trivial but the homotopy fibre was identified. The following two propositions adapt some of the trivialities of $\partial_d$ to our setting.
 
 \begin{prop} \label{proptrivialthentrivial}
  Fix $d\in \mathbb{Z}$ and let $\partial_d$ be the boundary map in (\ref{equnoneqeval}). Let $(P,\tin)$ be a Real principal $U(n)$-bundle of class $(2d,0,\dots,0)$ over a Real surface of type $(g,r,a)$. Let
  \[\partial_P\colon \Omega B \rightarrow B\g^*((g,r,a);(2d,0,\dots,0))\] be the boundary map of the evaluation fibration in (\ref{equevaluationfibration}). 
  For a prime $q$, if $\partial_d$ is (q-locally) trivial then
  \begin{enumerate}
  \item if $r>0$ then $\partial_P$ is (q-locally) trivial;
  \item if $r=0$ then the composition 
  \[O(n)\hookrightarrow U(n) \xrightarrow{\partial_P}  B\g^*((g,r,a);(2d,0,\dots,0))\]
  is (q-locally) trivial.
  \end{enumerate}
 \end{prop}
 \begin{proof}
 The key will be to compare both maps to another evaluation boundary map involving the  $\mathbb{Z}_2$-space $Y=(S^2\vee S^2,\sw)$.  We note that components of $\Map_{\mathbb{Z}_2}^*(Y,BU(n))$ are classified by even integers.
 
 Let $S^2\xrightarrow{i_1} S^2 \vee S^2=Y$ be the inclusion onto the left factor, and note that this is not a $\mathbb{Z}_2$-map.  The following diagram commutes
 \begin{equation} \label{eqntheriaultgeneralisation} \begin{gathered}\xymatrix{O(n)\ar[r]^-{\overline{\partial}_{2d}} \ar@{^{(}->}[d] & \Map_{\mathbb{Z}_2}^*(Y,BU(n);2d)\ar[d]^{i_1^*} \ar[r] & \Map_{\mathbb{Z}_2}(Y,BU(n);2d) \ar[d] \ar[r] & BO(n)\ar@{^{(}->}[d] \\
U(n) \ar[r]^-{\partial_d} & \Map^*(S^2,BU(n);d) \ar[r] & \Map(S^2,BU(n);d) \ar[r]& BU(n).}\end{gathered}\end{equation}

Now, there is an inverse to $i_1^*$ which sends a map $f$ in $\Map^*(S^2,BU(n);d)$ to the composition 
\[S^2 \vee S^2 \xrightarrow{f\vee f} BU(n)\vee BU(n) \xrightarrow{\id\vee \sigma_{BU(n)}}BU(n) \vee BU(n) \xrightarrow{\mbox{\footnotesize fold}} BU(n)\]
which is $\mathbb{Z}_2$-equivariant because the involution on $S^2\vee S^2$ swaps the factors.
Note that the map induced on the unpointed mapping spaces does not have an inverse because the basepoint of $Y$ must land in $BO(n)$. We conclude that if $\partial_d$ is $q$-locally trivial then so is $\overline{\partial}_{2d}$.  

Let $q\colon X \rightarrow Y$ be the map that collapses the $1$-skeleton of the Real surface $(X,\sigma)$. We obtain the following commutative diagram
 \[ \xymatrix{O(n)\ar[r]^-{\overline{\partial}_{2d}} \ar[d]^-f& \Map_{\mathbb{Z}_2}^*(Y,BU(n);2d)\ar[d]^{q^*} \ar[r] & \Map_{\mathbb{Z}_2}(Y,BU(n);2d) \ar[d] \ar[r] & BO(n)\ar[d] \\
\Omega B \ar[r]^-{\partial_P} & \Map_{\mathbb{Z}_2}^*(X,BU(n);P) \ar[r] & \Map_{\mathbb{Z}_2}(X,BU(n);P) \ar[r]& B}\]
The map $f$ is an equivalence if $r>0$ and is the inclusion $O(n)\hookrightarrow U(n)$ otherwise.  Since $\overline{\partial}_{2d}$ is (q-locally) trivial, the result follows.
\end{proof}

\begin{prop} \label{propcthentrivial} Let $p$ be a prime such that $p\nmid d$ and let $(P,\tin)$ be a Real principal $U(p)$-bundle of class $(2d,0,\dots,0)$ over a Real surface of type $(g,r,a)$.
  Let 
  \[\partial_P\colon \Omega B \rightarrow B\g^*((g,r,a);(2d,0,\dots,0))\]
  be the boundary map of the evaluation fibration. Then
  \begin{enumerate}
  \item if $r>0$ then $\partial_P$ is $p$-locally trivial;
  \item if $r=0$ then the composition 
  \[O(p)\hookrightarrow U(p) \xrightarrow{\partial_P}  B\g^*((g,r,a);(2d,0,\dots,0))\]
  is $p$-locally trivial.
  \end{enumerate}
 \end{prop}

\begin{proof}
We assume that that $p \nmid d$ is a prime and that all spaces and maps are localised at $p$. Let $Y=(S^2\vee S^2,\sw)$ be as above, then there is a homotopy commuting diagram
\begin{equation}\label{eqncreal} \begin{gathered}
\xymatrix{ O(p)\ar[rrr]^{\overline\partial_{2d}} \ar@{^{(}->}[d]^{i} & & & \Mappe(Y,BU(n);2d) \ar[d]^\simeq \\
U(p)\ar[rrr]^{\partial_d} \ar@{=}[d] & & &\Omega U(p)_0 \\
U(p)\ar[rrr]^{\partial_1} \ar[d]^e & & &\Omega U(p)_0 \ar[u]^d \ar[d]^{(\Omega e)_0}\\
\prod_{i=0}^{p-1} S^{2i+1}\ar[r]^-{\mbox{\footnotesize proj}} & S^{2p-1} \ar[r]^\alpha &\Omega S^3 \ar[r]^-{\mbox{\footnotesize incl}} & \prod_{j=1}^{p-1} \Omega S^{2j+1}
}
\end{gathered}
\end{equation}
where the top square is from diagram (\ref{eqntheriaultgeneralisation}) and the bottom two squares are found in \cite{tggoverriemannsurfaces}; specifically in Proposition $4.1$ and the proof of Theorem $1.1$ $(b)$ and $(c)$.

The map $d$-th power map $d\colon \Omega U(p)_0\rightarrow \Omega U(p)_0$ is a homotopy equivalence because $p\nmid d$. Furthermore, the maps $e$ and $(\Omega e)_0$ are homotopy equivalences provided in \cite{Serrehomotopygroups}. Now for $p\neq 2$ prime, there is a $p$-local homotopy equivalence
\[SO(p)\simeq_p \prod_{i=1}^{\frac{p-1}{2}} S^{4i-1} \]
and furthermore the inclusion $O(p)\hookrightarrow U(p)$ is in fact the inclusion of these factors into $\prod_{i=0}^{p-1} S^{2i+1}$.  We conclude that the composition
\begin{equation}\label{eqntrivialmapc} \chi\colon O(p)\hookrightarrow U(p) \rightarrow \prod_{i=0}^{p-1} S^{2i+1} \xrightarrow{\mbox{\footnotesize proj}} S^{2p-1}\end{equation}
is nullhomotopic and therefore so is $\overline\partial_{2d}$.

For $p=2$, the space $O(2)$ is homeomorphic to $S^1\amalg S^1$. Since $\chi$ in $(\ref{eqntrivialmapc})$ has target space $S^3$, we conclude that $\chi$ and hence $\overline \partial_{2d}$ are nullhomotopic in this case too.  The result then follows in a similar way to the last paragraph in the proof of Proposition \ref{proptrivialthentrivial}.
\end{proof}
\begin{proof}[Proof of Theorem \ref{thme} $(1a)$ and $(2a)$]\label{proofofthmea}
Theorem \ref{thm:tgghomotopydecomp} and Proposition \ref{proptrivialthentrivial} immediately obtain $(1a)$. Similarly Theorem \ref{thm:tgghomotopydecomp} and Proposition \ref{proptrivialthentrivial} obtain $(2a)$ when $p\mid d$, and Proposition \ref{propcthentrivial} then gives the remaining case when $p\nmid d$.
\end{proof}

\subsubsection{Case: \texorpdfstring{$(0,0,1)$}{(001)}}
We restrict to analysing gauge groups above Real surfaces of type $(0,0,1)$. Fix an even integer $c$ then we wish to analyse the boundary map $\partial_c$ of the evaluation fibration.  

For a $\mathbb{Z}_2$-space $A$, let $\overline{\Delta}\colon A\rightarrow A\times A$ be the composition
\begin{equation}\label{eqndeltaoverline}A \xrightarrow{\Delta} A\times A \xrightarrow{\id\times \sigma_{A}} A\times A.\end{equation}
Let $u$ be the composition
\[B\g^*((0,0,1);c) \xrightarrow{\simeq} \Mappe(S^2,BU(n);c) \xrightarrow{\tilde u} \Map^{*2}(D^2,BU(n)) \xrightarrow{\simeq} U(n)\] where $\tilde u$ restricts to the upper hemisphere of $(S^2, -\id)$ and forgets about equivariance except at $*$ and $\sigma(*)$. The last equivalence follows since $D^2$ with two points identified is homotopy equivalent to $S^1$. The maps $u$ and $\overline \Delta$ are the same as in equation (\ref{diagram:pullbackmap}) and they fit into the following commutative diagram \small
\[ \xymatrix{ U(n) \ar[r]^-{\partial_c} \ar[d]^-{\overline{\Delta}} & B\g^*((0,0,1);c)\ar[d]^u \ar[r] & B\g((0,0,1);c) \ar[d] \ar[r] & BU(n) \ar[d]^{\overline{\Delta}} \\ U(n)\times U(n)  \ar[r]^-\zeta & U(n) \ar[r] & \Map(D^2,BU(n)) \ar[r]^-{\ev_2}& BU(n)\times BU(n) } \]
\normalsize where $\ev_2$ evaluates at two antipodal points on the boundary of $D^2$ and $\zeta$ is defined via this diagram. 

Since $D^2$ is contractible, the map $\ev_2$ is homotopic to the diagonal map \[\Delta \colon BU(n) \rightarrow BU(n)\times BU(n).\] Therefore, the map $\zeta$ is homotopic to the map defined by $(A,B)\mapsto AB^{-1}$. Let $f\colon U(n)\rightarrow U(n)$ be defined as $f(A)=AA^t$ and we conclude that $u\partial_c \simeq f$. 

After localising the map $f$ at a prime $p\neq 2$, we can yield the composition
\begin{equation} \label{eqnlocalisedfwithprojection}
 SO(n)\times U(n)/SO(n) \xrightarrow{f} SO(n)\times U(n)/SO(n) \xrightarrow{p_2} U(n)/SO(n)
\end{equation}
 where $p_2$ is the projection map. Recall the map $\varphi$ from equation (\ref{eq:symmetricmatrixfactoring}) and compare with $f$. For $p\neq 2$, we showed that $\varphi$ is a $p$-local homotopy equivalence and we conclude that restricting the composition (\ref{eqnlocalisedfwithprojection}) to the factor $U(n)/SO(n)$ also obtains a $p$-local homotopy equivalence. We have shown the following proposition. 

\begin{prop}\label{propunpointed001map}
Let $n$ be odd, then localised at a prime $p\neq 2$ the following composition is a homotopy equivalence
\[U(n)/SO(n)\hookrightarrow U(n) \xrightarrow{\partial_c} B\g^*((0,0,1);c) \xrightarrow{u} U(n) \rightarrow U(n)/SO(n).\]
\qed
\end{prop}
With this proposition, we have enough ammunition to prove Theorem \ref{thme} $(1b)$ and $(2b)$. 

\begin{proof}[Proof of Theorem \ref{thme} $(1b)$ and $(2b)$]\label{proofoftheoreme34} We first prove part $(1b)$. Localise at a prime $p\neq 2$ such that $p\nmid n$ and reconsider the fibration sequence 
\[\g((0,0,1);c)\rightarrow SO(n)\times U(n)/SO(n)\xrightarrow{\partial_c} B\g^*((0,0,1);c).\]
By Proposition \ref{propunpointed001map} the factor $U(n)/SO(n)$ retracts off $B\g^*((0,0,1);c)$ and by Proposition \ref{proptrivialthentrivial} $(2)$ the factor $SO(n)$ retracts off $\g((0,0,1);c)$ under a lift \[l\colon SO(n)\rightarrow \g((0,0,1);c)\] of the inclusion $SO(n)\hookrightarrow U(n)$.  Then the composition
\begin{multline*} SO(n)\times \Omega^2(U(n)/SO(n)) \xrightarrow{l\times \id}\g((0,0,1);c) \times \Omega^2(U(n)/SO(n)) \\ \xrightarrow{\mbox{ \footnotesize {`action'}}}\g((0,0,1);c) \end{multline*}
 is homotopy equivalence and the result follows. The proof of part $(2b)$ is similar.
\end{proof}

\subsubsection{Case: \texorpdfstring{$(1,0,1)$}{(101)}}
We now analyse unpointed gauge groups above a Real surface $(T,\tau)$ of type $(1,0,1)$. We use a similar method to the $(0,0,1)$ case and adopt some of its notation.

As in the proof of Theorem \ref{thmtorusantipointsplitting}, let $u' \colon B\g^*((1,0,1);c) \rightarrow \Map^*(C,BU(n))$ be the map that forgets about equivariance and restricts to the upper half of $(T,\tau)$ which is homeomorphic to a cylinder $C$. Let $\overline \Delta$ be as in (\ref{eqndeltaoverline}) then we obtain the following diagram \small
\[ \xymatrix{ U(n) \ar[r]^-{\partial_c} \ar[d]^-{\overline{\Delta}} & B\g^*((1,0,1);c)\ar[d]^{u'} \ar[r] & B\g((1,0,1);c) \ar[d] \ar[r] & BU(n) \ar[d]^{\overline{\Delta}} \\ U(n)\times U(n)  \ar[r]^-{\zeta'} & \Map^{*2}(C,BU(n)) \ar[r] & \Map(C,BU(n)) \ar[r]^-{\ev_2}& BU(n)\times BU(n) } \] \normalsize
where $\ev_2$ is another double evaluation map; viewing $C$ as a sub-complex of $(T,\tau)$, the map $\ev_2$ evaluates at the basepoint $*_1$ and its image under the involution $\tau(*_1)$. Again, the map $\zeta'$ is defined via this diagram. 

As in the previous case, we aim to study the homotopy type of the map $\overline \Delta \zeta'$. However, it is not immediately clear on the homotopy type of the `boundary' map $\zeta'$. We note that $C\simeq S^1$ under a deformation retract fixing $*_1$ and taking $\tau(*_1)$ to $*_1$. Therefore, if we let $LU(n)$ be the free loop space of $U(n)$, we deduce that there is a homotopy commutative diagram
\[\xymatrix{\Map(C,BU(n))\ar[d]^{\simeq} \ar[r]^{\ev_2} & BU(n)\times BU(n)\\
LBU(n)\ar[r]^{\ev} & \ar[u]^\Delta BU(n)
}\]
where $\ev$ evaluates at the basepoint $*_1$ and $\Delta$ is the diagonal map.  Given that $\Delta \ev$ is a composition, we obtain the following homotopy commutative diagram
\[ \xymatrix{ & U(n) \times U(n) \ar[d]^{\zeta'} \ar@{=}[r] & U(n)\times U(n) \ar[d]^{\tilde \zeta} \\
U(n) \ar@{=}[d] \ar[r]^-{h'} & \Map^{*2}(C,BU(n)) \ar[d] \ar[r]^-h& U(n) \ar[d]^{*} \\
U(n) \ar[r] & LBU(n) \ar[d]^{\Delta \ev} \ar[r]^\ev& BU(n) \ar[d]^\Delta \\
& BU(n) \times BU(n) \ar@{=}[r]& BU(n) \times BU(n)
}\]
where the map $*$ is the inclusion of the homotopy fibre, which is nullhomotopic. The middle square is a homotopy pullback and hence the maps $h$ and $h'$ are defined using this diagram.  

By the triviality of the middle right vertical, there is a right homotopy inverse $i$ to $h$ and a left inverse $q$ to $h'$. Therefore the space $\Map^{*2}(C,BU(n))$ is homotopy equivalent to the product $U(n)\times U(n)$.\footnote{Of course, this can be seen directly by studying the homotopy type of $C$.} Therefore the homotopy type of \[\zeta'\colon U(n)\times U(n) \rightarrow \Map^{*2}(C,BU(n))\] can be determined by studying $q\zeta'$ and $h\zeta'$.  It is clear that $q\zeta'\sim *$ and $h\zeta'\sim \tilde \zeta$. However, $\tilde \zeta$ is the same as the map $\zeta\colon U(n)\times U(n) \rightarrow U(n)$ in Case $(0,0,1)$ and therefore it homotopic to the map $(A,B)\mapsto AB^{-1}$. 

We conclude that $\zeta'$ is homotopic to a map 
\begin{align*}U(n)\times U(n) & \rightarrow U(n) \times U(n) \\
(A,B) & \mapsto (I_n,AB^{-1}).\end{align*}

\begin{proof}[Proof of Theorem \ref{thme} $(1c)$ and $(2c)$]\label{proofoftheoreme56}
We first prove part $(1c)$. Let $p\neq 2$ be a prime with $p \nmid n$. Then localised at $p$, in the same way as Proposition \ref{propunpointed001map} we see that the factor $U(n)/SO(n)$ in \[U(n)\simeq_p U(n)/SO(n)\times SO(n)\] retracts off $B\g^*((1,0,1);c)$ via
\[U(n)/SO(n)\hookrightarrow U(n) \xrightarrow{\partial_c} B\g^*((1,0,1);c) \xrightarrow{u'} U(n) \rightarrow U(n)/SO(n).\]
Additionally by Proposition \ref{proptrivialthentrivial} $(2)$, the factor $SO(n)$ retracts off the gauge group $\g((1,0,1);c)$.  We find the required homotopy equivalence as in the proof of Theorem \ref{thme} $(1b)$.  The proof of $(2c)$ is similar.
\end{proof}

\subsection{The Quaternionic Case}\label{sectionthequaternioniccase}

From herein we restrict to the Quaternionic case. Again, our method of attack will be to study some mapping spaces related to these gauge groups.  In fact these mapping spaces are the same as in the Real case except $BU(2n)$ is endowed with an involution so that \[(EU(2n),\tilde\varsigma_{Q})\rightarrow (BU(2n),\varsigma_{Q})\] is a universal Quaternionic bundle. 
Recall that in the Real case, the involution $\varsigma$ was induced by complex conjugation $\hat\varsigma\colon U(n) \rightarrow U(n)$.  In this case, the involution $\varsigma_Q$ is induced from the homomorphism
\setlength{\arraycolsep}{2pt}
\[\begin{array}{rccc}
\hat\varsigma_Q\colon & U(n) &\rightarrow& U(n) \\ 
& A & \mapsto &J^{-1} A J
\end{array}\]
\setlength{\arraycolsep}{6pt}
where
\[J=\left( \begin{array}{ccccccc}
   0 & 1 & 0 & 0 & \cdots & 0 & 0\\
  -1 & 0 & 0 & 0 & \cdots & 0 & 0\\
   0 & 0 & 0 & 1 & \cdots & 0 & 0\\
   0 & 0 & -1& 0 & \cdots & 0 & 0\\
   \vdots & \vdots & \vdots& \vdots & \ddots & \vdots & \vdots\\
   0 & 0 & 0 & 0 & \cdots & 0 & 1\\
   0 & 0 & 0 & 0 & \cdots & -1 & 0\\
  \end{array}\right)\hspace{-5pt}\begin{array}{l} \vphantom{\vdots}\\ \\ \\ \\ \\ \\ .\end{array}\]
Most of the results in the Real case come from geometric properties of $(X,\sigma)$, hence we will see that these results transfer to the Quaternionic setting without too much hassle. Furthermore, since $(BG)^{\varsigma_Q}=BSp(n)$ we will see that a number of results will be easier to prove due to the high connectivity of $BSp(n)$.

For the $\mathbb{Z}_2$-space $(BU(2n),\varsigma_Q)$ as above, we write 
\[\Mapq(X,BU(2n)):=\Mape(X,BU(2n))\]
to distinguish from the Real case and use similar notation for the pointed cases. Now let $\overline X$ be as in the preamble to Theorem \ref{thmbaird}, and we state the Quaternionic analogue of Theorem \ref{thmbaird}.
\begin{theorem}\label{thmquatbaird}
Let $(P,\tin)$ be a Quaternionic principal $U(2n)$-bundle of class $c$ over a Real surface $(X,\bin)$ of type $(g,r,a)$. Then there are homotopy equivalences
 \begin{enumerate}
   \item $B\gq(P,\tin)\simeq \Mapq(X,BU(2n);P)$;
   \item $B\gq^*(P,\tin)\simeq \Mappq(X,BU(2n);P)$;
   \item $B\gq^{*(r+a)}(P,\tin)\simeq \Mapq^{*(r+a)}(X,BU(2n);c)\simeq \Mappq(\overline X,BU(2n);P) .$
  \end{enumerate}
where on the right hand side we pick the path component of $\Mapq(X,BU(n))$ that induces $(P,\tilde\sigma)$.
\end{theorem}
   We now sketch the proofs for the results in Section \ref{sectionmainresultsforquatbundles}.
    \begin{proof}[Proof of \ref{propquatcomponentsequivalence}]
   We use the action of $\pi_2(BU(2n))$ on $[(X,\sigma),(BU(2n),\varsigma_Q)]_{\mathbb{Z}_2}$ as presented in the proof of Proposition \ref{proprealpointedpathequal}.
 \end{proof}
 
  As in the Real case, the lack of $\pi_2(BU(2n))$ action means that we cannot provide an analogue for $B\gq(P,\tilde\sigma)$.
  \begin{proof}[Proof of \ref{propquatunpointedimmediate}]
     The idea is to tensor the Quaternionic bundle $(P,\tin)$ with a Real $U(1)$-bundle $\pi_Q\colon (Q,\tau)\rightarrow (X,\bin)$ of class $(2,0,\dots,0)$.  The required isomorphism of gauge groups is then defined as in the proof of Proposition \ref{propunpointedimmediatefromnonequiv}.   
 \end{proof}
 We sketch the proofs for the results related to homotopy decompositions of the gauge groups.
 \begin{proof}[Proof of \ref{thmqa} and \ref{thmqb}]
 The proof is similar to those in Sections \ref{subsectionintegraldecompositions}--\ref{subsectionr0a2} except that in this case $BU(2n)^{\varsigma_Q}=BSp(n)$. We recall that decompositions involving fixed circles in the Real case needed to be handled delicately, but this does not occur in the Quaternionic case due to the high connectivity of $BSp(n)$. 
\end{proof}
 Our aim is to now prove Theorem \ref{thmqc}, using a similar method to that of Theorem \ref{thmsphereantipointsplitting}. Localised at a prime $p\neq 2$ and for $n$ odd, we obtained a $p$-local decomposition in the Real case due to the fact that the $p$-local homotopy equivalence 
\begin{align*} U(n)/O(n) & \rightarrow U(n)/O(n) \\
AO(n)& \mapsto AA^t O(n)
\end{align*}
factored through $B\g^*((0,0,1);0)$.  We shall see that a similar map involving $U(2n)/Sp(n)$ also factors through the Quaternionic analogue of this gauge group.  

  Let \[u \colon B\gq^*((0,0,1);0) \rightarrow \Map^{*2}(D^2,BU(2n))\] be the map that restricts to the upper hemisphere of $(S^2,-\id)$ and forgets about equivariance considering the image as landing in $\Map^{*2}(D^2,BU(2n))$.  Let \[r\colon B\gq^*((0,0,1);0) \rightarrow \Mappe((S^1\vee S^1,\sw),BU(2n))\] be the map restricting to the $1$-skeleton of $(S^2,-\id)$. We obtain a similar homotopy commuting diagram to diagram (\ref{diagram:pullbackmap})

\begin{equation*} \begin{gathered}\xymatrix{Q \ar@/_/[ddr]\ar@/^/[drr] \ar@{.>}[dr]\\&B\gq^*((0,0,1);0)\ar[r]^-{u} \ar[d]^-r & U(2n) \ar[d]^-{\Delta^{-1}} \\ & U(2n) \ar[r]^-{{\Delta^Q}} & U(2n)\times U(2n).}\end{gathered}\end{equation*}
where $ \Delta^Q$ is the map $A\mapsto (A, \hat\varsigma_Q A)$.  Here, $Q$ is the strict pullback of the diagram 
\[ U(2n)\xrightarrow{\Delta^Q} U(2n)\times U(2n) \xleftarrow{\Delta^{-1}} U(2n)\]
and $B\gq^*((0,0,1);0)$ is the homotopy pullback of the same diagram. Once again, we aim to show that $Q$ retracts off $B\gq^*((0,0,1);0)$.
\begin{lemma}\label{lemmaQhomeomorphictoU(n)sp(n)}
 The pullback $Q$ is homeomorphic to $U(2n)/Sp(n)$.
\end{lemma}
\begin{proof} This is essentially the same proof as Lemma \ref{lemmaQhomeomorphictoU(n)/O(n)}, but we must elaborate on the details for surjectivity of the map 
\begin{align*}
f'\colon  U(2n)/Sp(n) & \rightarrow  Q \\
  ASp(n) & \mapsto  A \hat\varsigma_Q(A)^{-1}.
\end{align*}
 It can be shown that a matrix $A$ is in $Q$ if and only if $AJ$ is skew-symmetric and hence due the Youla Lemma \cite{youlaanormalform}, there is a unitary matrix $P$ such that $AJ=PJP^t$. Therefore
\[A=PJP^tJ^{-1}=P(J^{-1}\overline PJ)^{-1}=f'(PSp(n))\]
 and the result follows.
\end{proof}

Similar to the map in (\ref{eq:symmetricmatrixfactoring}), we obtain the following composition
\begin{equation}\label{eqnquatmatrixfact}\varphi \colon U(2n)/Sp(n) \xrightarrow{f'} Q \rightarrow B\gq^*((0,0,1);0) \xrightarrow{r} U(2n) \xrightarrow{q} U(2n)/Sp(n)\end{equation}
where $q$ is the quotient map.  The map $\varphi$ sends an element $ASp(n)$ to the element $A\hat\varsigma_Q(A)^{-1}Sp(n)$. It was shown in \cite{Harrisonthehomotopygroupsoftheclassicalgroups} that the related map 
\begin{equation}\label{eqnsunspnmap}\begin{gathered}\begin{aligned} s'\colon SU(2n)/Sp(n)& \rightarrow SU(2n)/Sp(n) \\ ASp(n)& \mapsto A\hat\varsigma_Q(A)^{-1}Sp(n)\end{aligned}\end{gathered}\end{equation}
is a homotopy equivalence when localised at a prime $p\neq 2$.

It is clear that there are analogue statements to Lemmas \ref{lemmaunsonsunsons1} and \ref{lemmathereexistmaps} and Proposition \ref{thmidentifyF}.
\begin{lemma}\label{lemmaquatunsonsunsons1}
There is a homotopy equivalence
\[ \eta\colon  U(2n)/Sp(n)\times S^1\xrightarrow{\simeq} U(2n)/Sp(n) . \]
\qed
\end{lemma}

\begin{lemma}\label{lemmaquatthereexistmaps}
There exist maps 
\begin{align*}s''\colon SU(n)/SO(n) & \rightarrow SU(n)/SO(n) \\
 s'\colon S^1 & \rightarrow S^1
\end{align*}
 such that the following is a homotopy commuting square
\[\xymatrix{SU(2n)/Sp(n)\times S^1 \ar[r]^{s''\times s'} \ar[d]^{\eta}& SU(2n)/Sp(n)\times S^1 \ar[d]^\eta\\
U(2n)/Sp(n)  \ar[r]^s & U(2n)/Sp(n).} \]
Furthermore, $s''$ and $s'$ are $p$-local equivalences. \qed
\end{lemma}

\begin{prop}
Let $F$ be the homotopy fibre of the composition
\[B\gq^*((0,0,1);0)\xrightarrow{r} U(2n) \xrightarrow{q} U(2n)/Sp(n)\]
then for any prime $p\neq 2$, there is a $p$-local homotopy equivalence 
\[F\simeq_p \Omega (U(2n)/Sp(n)). \]
\qed
\end{prop}

\begin{proof}[Proof of Proposition \ref{thmqc} (1)]
For a prime $p\neq 2$, we have shown that there is a $p$-local section to the principal homotopy fibration
\[ \Omega^2 (U(2n)/Sp(n))\rightarrow \gq^*((0,0,1);0) \xrightarrow{\Omega (qr)} \Omega(U(2n)/Sp(n))\]
and the result follows.
\end{proof}

\begin{proof}[Proof of Theorem \ref{thmqc} (2) and (3)]  These follow using the same proofs as Theorems \ref{thmtorusantipointsplitting} and \ref{thm111pointed}.
\end{proof}

In the unpointed case, the theorems involving integral decompositions follow immediately from the Real case.
\begin{proof}[Proof of Theorem \ref{thmqd}]  The results presented in Section \ref{subsectionintegraldecompositions} do not depend on the fixed point set of the involution on $BU(n)$ and hence Theorem \ref{thmqd} follows immediately.
\end{proof}
We proceed to prove the Quaternionic analogues of Section \ref{subsectionanalysingboundary}. Let
\[B=\begin{cases}
BSp(n) & \mbox{if } r>0 \\
BU(2n) & \mbox{otherwise}
\end{cases}\]
and recall the evaluation fibration
\begin{equation}\label{eqquatevaluationfibration} \Omega B \xrightarrow{\partial_P} \Map_Q^*(X, BU(n);P)\rightarrow \Map_Q(X, BU(n);P) \rightarrow B.\end{equation}
The following proposition can be proven using the same method as Proposition \ref{proptrivialthentrivial}.
\begin{prop}\label{propquattrivialthentriv}
Fix $d \in \mathbb{Z}$ and let $\partial_d$ be the boundary map in (\ref{equnoneqeval}). Let 
  \[\partial_P\colon \Omega B \rightarrow B\gq^*((g,r,a);2d)\] be the boundary map of the evaluation fibration as in (\ref{eqquatevaluationfibration}). 
  If $\partial_d$ is (q-locally) trivial then
  \begin{enumerate}
  \item if $r>0$ then $\partial_P$ is (q-locally) trivial;
  \item if $r=0$ then the composition 
  \[Sp(n)\hookrightarrow U(2n) \xrightarrow{\partial_P}  B\gq^*((g,r,a);2d)\]
  is (q-locally) trivial. \qed
  \end{enumerate}
\end{prop}

\begin{proof}[Proof of Theorem \ref{thmqe} (1)]
Let $p$ be a prime such that $p \nmid 2n$, then by Theorem \ref{thm:tgghomotopydecomp}, the map $\partial_{2n}$ is $p$-locally trivial. The result then follows from Proposition \ref{propquattrivialthentriv}.
\end{proof}

\begin{proof}[Proof of Theorem \ref{thmqe} (2) and (3)]
The proof is similar to the proofs of Theorem \ref{thme} $(1b)$ and $(1c)$.  We do require that $p\neq 2$ but this is automatic with the assumption that $p\nmid 2n$.
\end{proof}

\section{Tables of Homotopy Groups}\label{appendixa}
We present homotopy groups of the $(r+a)$-pointed and unpointed gauge groups.  We only present these for the trivial components, that is,
\begin{itemize}
\item $(c,w_1,\dots w_r)=(0,0,\dots,0)$ for Real bundles and;
\item $c=0$ for Quaternionic bundles;
\end{itemize}
with the understanding that results can be obtained for different components using the results in Section \ref{chapterintroduction}.  Specifically, in the $(r+a)$-pointed case, we can obtain results for all of the components using Propositions \ref{proprealpointedpathequal} and \ref{propquatcomponentsequivalence} and, in the unpointed case, we can obtain results for some of the different components using Propositions   \ref{propunpointedimmediatefromnonequiv}, \ref{propunpointedcomponentequiv},  and \ref{propquatunpointedimmediate}.

We first recall the status of the calculation of the homotopy groups before this paper, that is, we present the low dimensional homotopy groups from \cite{biswasstablevectorbundles} in Table \ref{tablebiswas}.

\renewcommand{\arraystretch}{1.8}
\begin{center}
\captionsetup{type=table}
\captionof{table}{Results of \cite{biswasstablevectorbundles} -- the low dimensional homotopy groups of rank~$n$ gauge groups above a Real surface of type $(g,r,a)$. The entries in blue disagree with the author's results.}\label{tablebiswas}\setlength{\arraycolsep}{5pt}
$
\begin{array}{|c || c | c | c | c |}\hline 
 \mbox{ Real } & \pi_0(\g^{*(r+a)}(P,\tin)) & \pi_0(\g(P,\tin))& \pi_1(\g^{*(r+a)}(P,\tin))&\pi_1(\g(P,\tin))\\
\hline\hline 
 n>2& \mathbb{Z}^{g+a}\times (\mathbb{Z}_2)^r& \mathbb{Z}^g\times (\mathbb{Z}_2)^{r+1}& \textcolor{colour1}{\mathbb{Z}} & \textcolor{colour1}{\mathbb{Z}}\times(\mathbb{Z}_2)^r \\ \hline
 n=2& \mathbb{Z}^{g+a+r} & \mathbb{Z}^{g+r} \times \mathbb{Z}_2 & \mathbb{Z}& \mathbb{Z}^{r+1} \\ \hline
 n=1& \mathbb{Z}^{g+a} & \mathbb{Z}^g\times \mathbb{Z}_2  & 0 & 0\\ \hline \hline
 \begin{gathered}\begin{array}{c}
\mbox{Quat.} \\[-9pt]
\mbox{rank }2n
\end{array}\end{gathered} & \mathbb{Z}^{g+a} & \mathbb{Z}^g\times (\mathbb{Z}_2)^a& \textcolor{colour1}{\mathbb{Z}} & \textcolor{colour1}{\mathbb{Z}}\\ \hline
\end{array}
$
\end{center}
\setlength{\arraycolsep}{6pt}

\renewcommand{\arraystretch}{1.35}

From the results in Sections \ref{sectionmainresultsforrealbundles} and \ref{sectionmainresultsforquatbundles}, we can see that our homotopy decompositions usually contain factors involving $U(n)$, $O(n)$ and $Sp(n)$.  Due to Bott periodicity, it is easy to calculate some of the higher homotopy groups for high rank gauge groups. We present such results in Tables \ref{tablehighrank} and \ref{tablequathighrank} where $\eta$ is defined via
\[\eta=\eta(g,r,a)=\begin{cases}
1 & \mbox{if } r>0 \mbox{ and } a=1; \\
0 & \mbox{otherwise}.
\end{cases}\]
Some of the results in Table \ref{tablehighrank} are a consequence of localised homotopy equivalences and hence may provide incomplete information.  To highlight these localised results we use the following notation
\begin{itemize}
\item groups surrounded by $(-)_p$ are understood to have come from $p$-local homotopy equivalences where $p$ and the rank $n$ of the gauge groups satisfy the requirements of Theorems \ref{thmc} and \ref{thme}.
\end{itemize}

\begin{center}
\captionsetup{type=table}
\captionof{table}{{Homotopy groups for high rank gauge groups of Real bundles, that is, the homotopy groups $\pi_i$ when the rank $n>i+2$. The results in blue correspond to the top row in Table \ref{tablebiswas}. }}\label{tablehighrank}
$
\begin{array}{|c||c|c|} 
\hline
  & \g^{*(r+a)}(P,\tin)  & \g(P,\tin) \\
  \hline \hline
  \pi_{8j} & \textcolor{colour1}{\mathbb{Z}^{g-1}\times \mathbb{Z}_2^{r-1} \times (\mathbb{Z}^{1+a})_p\times (\mathbb{Z}_2^{1+\eta})_p}  & \textcolor{colour1}{\mathbb{Z}^{g-1}\times \mathbb{Z}_2^{r-1} \times (\mathbb{Z})_p\times (\mathbb{Z}_2^{1+\eta})_p}  \\
  \hline
  \pi_{8j+1} & \textcolor{colour1}{(\mathbb{Z}_2^{1+a})_p}   & \textcolor{colour1}{\mathbb{Z}_2^{r-1} \times (\mathbb{Z}_2^{2+\eta})_p}   \\
  \hline
   \pi_{8j+2} & \mathbb{Z}^{g+r-2} \times (\mathbb{Z}^{1+\eta})_p\times (\mathbb{Z}_2^{a})_p  &\mathbb{Z}^{g-1}\times \mathbb{Z}_2^{r-1} \times (\mathbb{Z})_p\times (\mathbb{Z}_2^{\eta})_p  \\
  \hline
  \pi_{8j+3} & (\mathbb{Z})_p &    (\mathbb{Z}^2)_p\\
  \hline
  \pi_{8j+4} & \mathbb{Z}^{g-1} \times (\mathbb{Z}^{1+a})_p   & \mathbb{Z}^{g-1}\times  (\mathbb{Z})_p \\
  \hline
  \pi_{8j+5} & 0 & 0 \\
  \hline
   \pi_{8j+6} &\mathbb{Z}^{g+r-2}\times (\mathbb{Z}^{1+\eta})_p & \mathbb{Z}^{g-1}\times (\mathbb{Z}^{1-\eta})_p\\
   \hline
    \pi_{8j+7} & \mathbb{Z}_2^{r-1} \times (\mathbb{Z})_p\times (\mathbb{Z}_2^{\eta})_p & \mathbb{Z}_2^{r-1} \times (\mathbb{Z}^2)_p\times (\mathbb{Z}_2^{\eta})_p\\
    \hline
\end{array}
$
\end{center}

Similarly, some of the results in Table \ref{tablequathighrank} are a consequence of localised homotopy equivalences and hence may provide incomplete information.  To highlight these localised results we use the following notation
\begin{itemize}
\item groups surrounded by $(-)_p$ are understood to have come from $p$-local homotopy equivalences where $p$ is prime and the rank $2n$ of the gauge groups satisfy the requirements of Theorems \ref{thmqc} and \ref{thmqe}.
\end{itemize}

\begin{center}
\captionsetup{type=table}
\captionof{table}{{Homotopy groups for high rank gauge groups of Quaternionic bundles, that is, the homotopy groups $\pi_i$ when the rank $2n>\frac{i+1}{4}$. The results in blue correspond to the bottom row in Table \ref{tablebiswas}. }}\label{tablequathighrank}
$
\begin{array}{|c||c|c|} 
\hline
  & \gq^{*(r+a)}(P,\tin)  & \gq(P,\tin) \\
  \hline \hline
 
  \pi_{8j} & \textcolor{colour1}{\mathbb{Z}^{g-1} \times (\mathbb{Z}^{1+a})_p}   & \textcolor{colour1}{\mathbb{Z}^{g-1}\times  (\mathbb{Z})_p} \\
  \hline
  \pi_{8j+1} &  \textcolor{colour1}{0} & \textcolor{colour1}{0} \\
  \hline
   \pi_{8j+2} &\mathbb{Z}^{g+r-2}\times (\mathbb{Z}^{1+\eta})_p & \mathbb{Z}^{g-1}\times (\mathbb{Z}^{1-\eta})_p\\
   \hline
    \pi_{8j+3} & \mathbb{Z}_2^{r-1} \times (\mathbb{Z})_p\times (\mathbb{Z}_2^{\eta})_p & \mathbb{Z}_2^{r-1} \times (\mathbb{Z}^2)_p\times (\mathbb{Z}_2^{\eta})_p\\
    \hline
     \pi_{8j+4} & \mathbb{Z}^{g-1}\times \mathbb{Z}_2^{r-1} \times (\mathbb{Z}^{1+a})_p\times (\mathbb{Z}_2^{1+\eta})_p  & \mathbb{Z}^{g-1}\times \mathbb{Z}_2^{r-1} \times (\mathbb{Z})_p\times (\mathbb{Z}_2^{1+\eta})_p  \\
  \hline
  \pi_{8j+5} & (\mathbb{Z}_2^{1+a})_p   & \mathbb{Z}_2^{r-1} \times (\mathbb{Z}_2^{2+\eta})_p   \\
  \hline
   \pi_{8j+6} & \mathbb{Z}^{g+r-2} \times (\mathbb{Z}^{1+\eta})_p\times (\mathbb{Z}_2^{a})_p  &\mathbb{Z}^{g-1}\times \mathbb{Z}_2^{r-1} \times (\mathbb{Z})_p\times (\mathbb{Z}_2^{\eta})_p  \\
  \hline
  \pi_{8j+7} & (\mathbb{Z})_p &    (\mathbb{Z}^2)_p\\
  \hline
\end{array}
$
\end{center}

Due to the properties of Bott periodicity, Table \ref{tablequathighrank} is a translation of Table \ref{tablehighrank}. We note that additional calculations can be made for the lower rank cases.  We point the reader to \cite[Section 3.2]{mimurahomotopytheoryofliegroups} where explicit homotopy groups of some of the relevant factors can be found.

We note that the author's results disagree with the $\mathbb{Z}$-summands coloured in blue in Table \ref{tablebiswas}.  In the pointed case, this $\mathbb{Z}$-summand arises in \cite{biswasstablevectorbundles} by studying a fibration arising from restricting the gauge group to the $1$-skeleton of the Real surface. 

For example, the corresponding fibration for a type $(g,r,0)$ Real surface is
\[\Omega^2 U(n)\rightarrow \g^*(P,\tin)\rightarrow \prod^{g} \Omega U(n)\times \prod^{r} \Omega O(n)\]
and we obtain the exact sequence
\[0\rightarrow \pi_2(\g^*(P,\tin))\xrightarrow{\nu} \mathbb{Z}^{g+r}\rightarrow \mathbb{Z}\xrightarrow{\mu} \pi_1(\g^*(P,\tin))\rightarrow 0.\]
The claim in \cite{biswasstablevectorbundles} is that the map $\mu$ can be thought in terms of the classification of bundles over $S^2\wedge X$. Further since $\mu$ is induced by a map that collapses the one skeleton of $X$, the map $\mu$ is essentially providing an identification of the second Chern class, and hence is an isomorphism.

The author agrees that this argument holds in the non-equivariant case. Indeed, if we consider $X$ as a Riemann surface we obtain that $S^2\wedge X$ is a wedge of spheres and then $\mu$ is induced by a map that collapses all but the top copy of $S^4$. 

However, we now demonstrate that $\pi_1(\g^*(P,\tin))$ cannot contain a $\mathbb{Z}$-summand, at least for the type $(0,1,0)$ case.  We assume that $\pi_1(\g^*(P,\tin))$ contains a $\mathbb{Z}$-summand, and that subsequently the map $\mu$ is an isomorphism.  Therefore $\nu$ is an isomorphism and we recall that it is induced by the map $r'$ which restricts to the $1$-skeleton of $X$.  The map~$r'$ fits into the following commutative diagram
\[ \xymatrix{ \g^*(P,\tin) \ar[r]^-{u'} \ar[d]^{r'} & \Omega \Map^*(D^2,BU(n)) \ar[d]^{r} \\ \Omega \Mappe((S^1,\id), BU(n);0) \ar[r]^-u & \Omega \Map^*(S^1, BU(n))} \]
where $u'$ is the map that forgets about equivariance and restricts to the upper hemisphere of $X$.

Now $u$ is homotopic to the inclusion $\Omega O(n)\hookrightarrow \Omega U(n)$ and hence by assumption the induced map 
\[u_*\nu=(ur')_*\colon \mathbb{Z}\cong\pi_2(\g^*(P,\tin))\rightarrow \pi_2(\Omega U(n))\cong \mathbb{Z}\]
is multiplication by $2$.  But $ru'$ is nullhomotopic because it factors through the contractible space $\Omega \Map^*(D^2,BU(n))$ and we obtain a contradiction.  We conclude that $\mu$ cannot be an isomorphism.

It remains to show that the other blue entries in Table \ref{tablebiswas} cannot contain $\mathbb{Z}$-summands.  However, these entries were obtained from the calculation in the pointed case and therefore we argue that these cannot contain $\mathbb{Z}$-summands either.

\providecommand{\bysame}{\leavevmode\hbox to3em{\hrulefill}\thinspace}
\providecommand{\MR}{\relax\ifhmode\unskip\space\fi MR }
\providecommand{\MRhref}[2]{%
  \href{http://www.ams.org/mathscinet-getitem?mr=#1}{#2}
}
\providecommand{\href}[2]{#2}

\end{document}